\documentclass[hidelinks,reqno,11pt]{article}
\usepackage{fullpage,amsmath,amssymb,enumerate,enumitem,mathtools,float,accents,chngcntr,hyperref,url,graphicx}
\usepackage{amsthm}
\usepackage{tikz,pgfplots}

\usepackage{algorithm, algorithmic}

\usepackage{caption}
\numberwithin{equation}{section}

\newtheorem{theorem}{Theorem}[section]
\newtheorem*{theorem*}{Theorem}
\newtheorem{lemma}[theorem]{Lemma}
\newtheorem*{lemma*}{Lemma}
\newtheorem{proposition}[theorem]{Proposition}
\newtheorem{corollary}[theorem]{Corollary}
\newtheorem{claim}[theorem]{Claim}

\theoremstyle{definition}
\newtheorem{definition}[theorem]{Definition}
\newtheorem{remark}[theorem]{Remark}

\DeclarePairedDelimiter{\abs}{\lvert}{\rvert}
\DeclarePairedDelimiter{\norm}{\lVert}{\rVert}
\DeclarePairedDelimiter{\paren}{(}{)}
\DeclarePairedDelimiter{\braces}{\lbrace}{\rbrace}
\DeclarePairedDelimiter{\inprod}{\langle}{\rangle}
\DeclarePairedDelimiter{\sqbracket}{[}{]}

\DeclarePairedDelimiter{\tripnorm}{\lvert\kern-0.25ex\lvert\kern-0.25ex\lvert}{\rvert\kern-0.25ex\rvert\kern-0.25ex\rvert}

\renewcommand{\P}{\mathbb{P}}

\newcommand{\E}{\mathbb{E}}
\newcommand{\R}{\mathbb{R}}
\newcommand{\C}{\mathbb{C}}
\newcommand{\Z}{\mathbb{Z}}

\newcommand{\cross}{\times}

\newlength{\dhatheight}

\newcommand{\st}{\ | \ }
\newcommand{\sign}{\textnormal{sign}}
\newcommand{\Tr}{\textnormal{Tr}}

\newcommand{\diag}{\textnormal{diag}}
\newcommand{\diam}{\textnormal{diam}}
\newcommand{\ACW}{\textnormal{ACW}}

\title{Phase Retrieval via Randomized Kaczmarz: Theoretical Guarantees}

\author{Yan Shuo Tan \footnote{Department of Mathematics, University of Michigan, \href{mailto:yanshuo@umich.edu}{yanshuo@umich.edu}.}  \quad\quad\quad Roman Vershynin \footnote{Department of Mathematics, University of California, Irvine, \href{mailto:rvershyn@uci.edu}{rvershyn@uci.edu}.}}

\begin{document}
	
\maketitle

\begin{abstract}
We consider the problem of phase retrieval, i.e. that of solving systems of quadratic equations. 
A simple variant of the randomized Kaczmarz method was recently proposed for phase retrieval, and it was shown numerically to have a computational edge over state-of-the-art Wirtinger flow methods. In this paper, we provide the first theoretical guarantee for the convergence of the randomized Kaczmarz method for phase retrieval. We show that it is sufficient to have as many Gaussian measurements as the dimension, up to a constant factor. Along the way, we introduce a sufficient condition on measurement sets for which the randomized Kaczmarz method is guaranteed to work. We show that Gaussian sampling vectors satisfy this property with high probability; this is proved using a chaining argument coupled with bounds on VC dimension and metric entropy.
\end{abstract}

\section{Introduction}

The phase retrieval problem is that of solving a system of quadratic equations
\begin{equation}	\label{eq: phase retrieval}
\abs{\inprod{a_i,z}^2} = b_i^2, \quad\quad i = 1,2,\ldots,m
\end{equation}
where $a_i \in \R^n$ (or $\C^n$) are known sampling vectors,
$b_i > 0$ are observed measurements, and $z \in \R^n$ (or $\C^n$) is the decision variable. This problem is well motivated by practical concerns \cite{Fienup1982} and has been a topic of study from at least the early 1980s.

Over the last half a decade, there has been great interest in constructing and analyzing algorithms with provable guarantees given certain classes of sampling vector sets. One line of research involves ``lifting'' the quadratic system to a linear system, which is then solved using convex relaxation ({\em PhaseLift}) \cite{Cand??s2013}. A second method is to formulate and solve a linear program in the natural parameter space using an anchor vector ({\em PhaseMax}) \cite{Goldstein2016,Bahmani2016a,Hand2016a}. Although both of these methods can be proved to have near optimal sample efficiency, the most empirically successful approach has been to directly optimize various naturally-formulated non-convex loss functions, the most notable of which are displayed in Table \ref{fig: loss functions}.

\begin{table}
	\centering
	\begin{tabular}{| c | c | c |}
		\hline
		Loss function & Name & Papers \\
		\hline
		$\displaystyle{f(z) = \sum_{i=1}^m \paren{\abs{a_i,z}^2-b_i^2}^2}$ & Wirtinger Flow objective & \cite{Candes2015,Sun2016} \\
		\hline
		$\displaystyle{f(z) = \sum_{i=1}^m \paren{\abs{a_i,z}-b_i}^2}$ & Amplitude Flow objective &\cite{Wang2017,Zhang2016} \\
		\hline
		$\displaystyle{f(z) = \sum_{i=1}^m \abs*{\abs{a_i,z}^2-b_i^2}}$ & Robust Wirtinger Flow objective &\cite{Eldar2014,Duchi2017,Davis2017} \\
		\hline
	\end{tabular}
	\caption{Non-convex loss functions for phase retrieval}
	\label{fig: loss functions}
\end{table}

These loss functions enjoy nice properties which make them amenable to various optimization schemes \cite{Sun2016,Eldar2014}. Those with provable guarantees include the prox-linear method of \cite{Duchi2017}, and various gradient descent methods \cite{Candes2015,Chen2015,Wang2017,Zhang2016,Davis2017}. Some of these methods also involve adaptive measurement pruning to enhance performance.

In 2015, Wei \cite{Wei2015} proposed adapting a family of randomized Kaczmarz methods for solving the phase retrieval problem. He was able to show using numerical experiments that these methods perform comparably with state-of-the-art Wirtinger flow (gradient descent) methods when the sampling vectors are real or complex Gaussian, or when they follow the coded diffraction pattern (CDP) model \cite{Candes2015}. He also showed that randomized Kaczmarz methods outperform Wirtinger flow when the sampling vectors are the concatenation of a few unitary bases. Unfortunately, \cite{Wei2015} was not able to provide adequate theoretical justification for the convergence of these methods (see Theorem 2.6 in \cite{Wei2015}).

In this paper, we attempt to bridge this gap by showing that the basic randomized Kaczmarz scheme used in conjunction with truncated spectral initialization achieves {\em linear convergence} to the solution with high probability, whenever the sampling vectors are drawn uniformly from the sphere\footnote{This is essentially equivalent to being real Gaussian because of the concentration of norm phenomenon in high dimensions. Also, one may normalize vectors easily.} $S^{n-1}$ and the number of measurements $m$ is larger than a constant times the dimension $n$.

It is also interesting to note that the basic randomized Kaczmarz scheme is exactly \textit{stochastic gradient descent} for the Amplitude Flow objective, which suggests that other gradient descent schemes can also be accelerated using stochasticity.

\subsection{Randomized Kaczmarz for solving linear systems}

The Kaczmarz method is a fast iterative method for solving systems of overdetermined linear equations that works by iteratively satisfying one equation at a time. In 2009, Strohmer and Vershynin \cite{Strohmer2009} were able to give a provable guarantee on its rate of convergence, provided that the equation to be satisfied at each step is selected using a prescribed randomized scheme.

Suppose our system to be solved is given by
\begin{equation} \label{eq: Ax = b}
Ax = b,
\end{equation}
where $A$ is an $m$ by $n$ matrix. Denoting the rows of $A$ by $a_1^T,\ldots,a_m^T$,
we can write \eqref{eq: Ax = b} as the system of linear equations
$$
\inprod{a_i,x} = b_i, \quad i=1,\ldots,m.
$$
The solution set of each equation is a hyperplane. 
The randomized Kaczmarz method is a simple iterative algorithm in which we {\em project the running approximation onto the hyperplane of a randomly chosen equation}. More formally, at each step $k$ we randomly choose an index $r(k)$ from $[m]$ such that the probability that $r(k) = i$ is proportional to $\norm{a_i}_2^2$, and update the running approximation as follows:
\begin{equation*}
x_k := x_{k-1} + \frac{b_{r(k)} - \inprod{a_{r(k)},x_{k-1}}}{\norm{a_{r(k)}}_2^2}a_{r(k)}.
\end{equation*}

Strohmer and Vershynin \cite{Strohmer2009} were able to prove the following theorem:

\begin{theorem}[Linear convergence for linear systems] \label{thm: linear convergence for linear systems}
	Let $\kappa(A) = \norm{A}_F / \sigma_{\min}(A)$. Then for any initialization $x_0$ to the equation \eqref{eq: Ax = b}, the estimates given to us by randomized Kaczmarz satisfy
	\begin{equation*}
	\E\norm{x_k-x}_2^2 \leq \paren*{1-\kappa(A)^{-2}}^k \norm{x_0 -x}_2^2.
	\end{equation*}
\end{theorem}
Note that if $A$ has bounded condition number, then $\kappa(A) \asymp 1/\sqrt{n}$.

\subsection{Randomized Kaczmarz for phase retrieval}

In the phase retrieval problem \eqref{eq: phase retrieval}, each equation 
$$
\abs{\inprod{a_i,x}} = b_i
$$
defines two hyperplanes, one corresponding to each of $\pm x$.
A natural adaptation of the randomized Kaczmarz update for this situation is then to project the running approximation to the \textit{closer hyperplane}. We restrict to the case where each measurement vector $a_i$ has unit norm, so that in equations, this is given by
\begin{equation}  \label{def: randomized Kaczmarz phase retrieval}
x_k := x_{k-1} + \eta_k a_{r(k)},
\end{equation} 
where
\begin{equation*}
\eta_k = \sign(\inprod{a_{r(k)},x_{k-1}}) b_{r(k)} - \inprod{a_{r(k)},x_{k-1}}.
\end{equation*}

In order to obtain a convergence guarantee for this algorithm, we need to choose $x_0$ so that it is close enough to the signal vector $x$. This is unlike the case for linear systems where we could start with an arbitrary initial estimate $x_0 \in \R^n$, but the requirement is par for the course for phase retrieval algorithms. Unsurprisingly, there is a rich literature on how to obtain such estimates \cite{Cand??s2013,Chen2015,Zhang2016,Wang2017}. The best methods are able to obtain a good initial estimate using $O(n)$ samples.

\subsection{Contributions and main result}

The main result of our paper guarantees the linear convergence of randomized Kaczmarz algorithm for phase retrieval for random measurements $a_i$ that are drawn independently and uniformly from the unit sphere.

\begin{theorem}[Convergence guarantee for algorithm] \label{thm: guarantee for algorithm}
	Fix $\epsilon > 0$, $0 < \delta_1 \leq 1/2$, and $0 < \delta,\delta_2 \leq 1$. There are absolute constants $C, c > 0$ such that if
	\begin{equation*}
	m \geq C(n\log(m/n) + \log(1/\delta)),
	\end{equation*}
	then with probability at least $1-\delta$, $m$ sampling vectors selected uniformly and independently from the unit sphere $S^{n-1}$ form a set such that the following holds: Let $x \in \R^n$ be a signal vector and let $x_0$ be an initial estimate satisfying $\norm{x_0-x}_2 \leq c\sqrt{\delta_1}\norm{x}_2$. Then for any $\epsilon > 0$, if
	\begin{equation*}
	K \geq 2(\log(1/\epsilon) + \log(2/\delta_2))n,
	\end{equation*}
	then the $K$-th step randomized Kaczmarz estimate $x_K$ satisfies $\norm{x_K-x}_2^2 \leq \epsilon \norm{x_0 - x}_2^2$ with probability at least $1-\delta_1-\delta_2$.
\end{theorem}

Comparing this result with Theorem \ref{thm: linear convergence for linear systems}, we observe two key differences. First, there are now two sources of randomness: one is in the creation of the measurements $a_i$, and the other is in the selection of the equation at every iteration of the algorithm. The theorem gives a guarantee that holds with high probability over both sources of randomness. Theorem \ref{thm: guarantee for algorithm} also requires an initial estimate $x_0$. This is not hard to obtain. Indeed, using the \textit{truncated spectral initialization} method of \cite{Chen2015}, we may obtain such an estimate with high probability given $m \gtrsim n$. For more details, see Proposition \ref{prop: initialization guarantee}.

The proof of this theorem is more nontrivial than the Strohmer-Vershynin analysis of randomized Kaczmarz algorithm for linear systems \cite{Strohmer2009}. We break down the argument in smaller steps, each of which may be of independent interest to researchers in this field.

First, we generalize the Kaczmarz update formula \eqref{def: randomized Kaczmarz phase retrieval} and define what it means to take a randomized Kaczmarz step with respect to any probability measure on the sphere $S^{n-1}$: we choose a measurement vector at each step according to this measure. Using a simple geometric argument, we then provide a bound for the expected decrement in distance to the solution set in a single step, where the quality of the bound is given in terms of the properties of the measure we are using for the Kaczmarz update (Lemma \ref{lem: expected decrement}).

Performing the generalized Kaczmarz update with respect to the uniform measure on the sphere corresponds to running the algorithm with unlimited measurements. We utilize the symmetry of the uniform measure to compute an explicit formula for the bound on the stepwise expected decrement in distance. This decrement is geometric whenever we make the update from a point making an angle of less than $\pi/8$ with the true solution, so we obtain linear convergence conditioned on no iterates escaping from the ``basin of linear convergence''. We are able to bound the probability of this bad event using a supermartingale inequality (Theorem \ref{thm: linear convergence, unlimited measurements}).

Next, we abstract out the property of the uniform measure that allows us to obtain local linear convergence. We call this property the \emph{anti-concentration on wedges} property, calling it $\ACW$ for short. Using this convenient definition, we can easily generalize our previous proofs for the uniform measure to show that all $\ACW$ measures give rise to randomized Kaczmarz update schemes with local linear convergence (Theorem \ref{thm: linear convergence, ACW measure}).

The usual Kaczmarz update corresponds running the generalized Kaczmarz update with respect to $\mu_A := \frac{1}{m}\sum_{i=1}\delta_{a_i}$. We are able to prove that when the $a_i$'s are selected uniformly and independently from the sphere, then $\mu_A$ satisfies the $\ACW$ condition with high probability, so long as $m \gtrsim n$ (Theorem \ref{thm: finite measurement sets satisfy ACW}). The proof of this fact uses VC theory and a {\em chaining argument}, together with metric entropy estimates.

Finally, we are able to put everything together to prove a guarantee for the full algorithm in Section \ref{sec: proof of guarantee}. In that section, we also discuss the failure probabilities $\delta$, $\delta_1$ and $\delta_2$, and how they can be controlled.

\subsection{Related work}

During the preparation of this manuscript, we became aware of independent simultaneous work done by Jeong and G{\"u}nt{\"u}rk. They also studied the randomized Kaczmarz method adapted to phase retrieval, and obtained almost the same result that we did (see \cite{Jeong2017} and Theorem 1.1 therein). In order to prove their guarantee, they use a stopping time argument similar to ours, but replace the ACW condition with a stronger condition called \textit{admissibility}. They prove that measurement systems comprising vectors drawn independently and uniformly from the sphere satisfy this property with high probability, and the main tools they use in their proof are hyperplane tessellations and a net argument together with Lipschitz relaxation of indicator functions.

After submitting the first version of this manuscript, we also became aware of independent work done by Zhang, Zhou, Liang, and Chi \cite{Zhang2016}. Their work examines stochastic schemes in more generality (see Section 3 in their paper), and they claim to prove linear convergence for both the randomized Kaczmarz method as well as what they called \textit{Incremental Reshaped Wirtinger Flow}. However, they only prove that the distance to the solution decreases in expectation under a single Kaczmarz update (an analogue of our Lemma \ref{lem: expected decrement} specialized to real Gaussian measurements). As we will see in our paper, this bound cannot be naively iterated.

\subsection{Notation}

Throughout the paper, $C$ and $c$ are absolute constants that can change from line to line.

\section{Computations for a single step}

In this section, we will compute what happens in expectation for a single update step of the randomized Kaczmarz method. It will be convenient to generalize our sampling scheme slightly as follows. When we work with a fixed matrix $A$, we may view our selection of a random row $a_{r(k)}$ as drawing a random vector according to the measure $\mu_A := \frac{1}{m}\sum_{i=1}^m \delta_{a_i}$. We need not restrict ourselves to sums of Diracs. For any probability measure $\mu$ on the sphere $S^{n-1}$, we define the random map $P = P_\mu$ on vectors $z \in \R^n$ by setting
\begin{equation} \label{eq: P definition}
P z := z + \eta a,
\end{equation}
where
\begin{equation} \label{eq: sign of P}
\eta = \sign(\inprod{a,z})\abs{\inprod{a,x}} - \inprod{a,z} \quad\text{and}\quad a \sim \mu.
\end{equation}
Note that as before, $x$ is a fixed vector in $\R^n$ (think of $x$ as the actual solution of the phase retrieval problem). We call $P_\mu$ the \emph{generalized Kaczmarz projection with respect to $\mu$}. Using this update rule over independent realizations of $P$, $P_1,P_2,\ldots$, together with an initial estimate $x_0$, gives rise to a \emph{generalized randomized Kaczmarz algorithm} for finding $x$: set the $k$-th step estimate to be
\begin{equation} \label{def: generalized Kaczmarz}
x_k := P_kP_{k-1}\cdots P_1 x_0.
\end{equation} 

Fix a vector $z \in \R^n$ that is closer to $x$ than to $-x$, i.e. so that $\inprod{x,z} > 0$, and suppose that we are trying to find $x$. Examining the formula in \eqref{eq: sign of P}, we see that $P$ projects $z$ onto the right hyperplane (i.e., the one passing through $x$ instead of the one passing through $-x$) if and only if $\inprod{a,z}$ and $\inprod{a,x}$ have the same sign. In other words, this occurs if and only if the random vector $a$ does \emph{not} fall into the region of the sphere defined by
\begin{equation} \label{eq: W definition}
W_{x,z} := \braces{v \in S^{n-1} \st \sign\paren{\inprod{v,x}} \neq \sign\paren{\inprod{v,z}}}.
\end{equation}
This is the region lying between the two hemispheres with normal vectors $x$ and $z$. We call such a region a \emph{spherical wedge}, since in three dimensions it has the shape depicted in Figure \ref{fig: geometry of W}.

\begin{figure}[h]
	\includegraphics[scale=0.6]{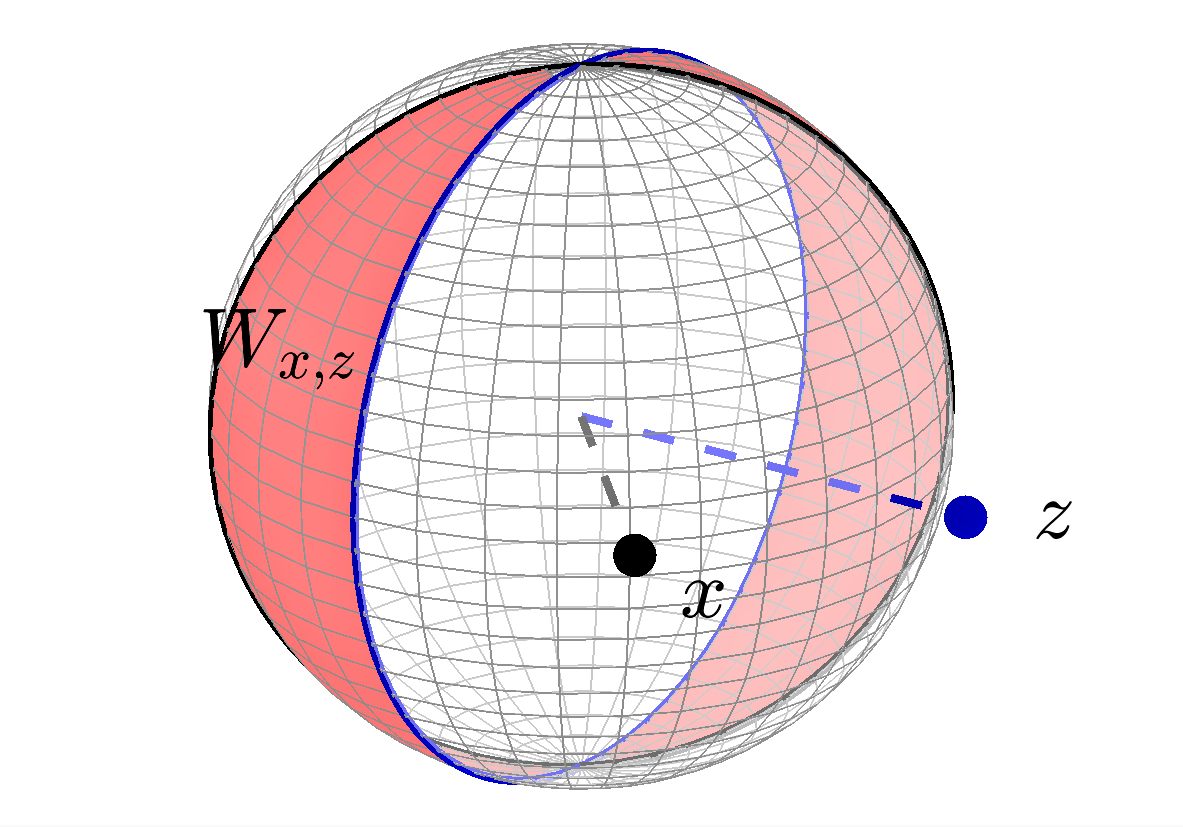}
	\centering
	\caption{Geometry of $W_{x,z}$}
	\label{fig: geometry of W}
\end{figure}

When $a \notin W_{x,z}$, we can use the Pythagorean theorem to write
\begin{equation} \label{eq: bound in A pt 1}
\norm{z-x}_2^2 = \norm{Pz-x}_2^2 + \inprod{z-x,a}^2.
\end{equation}
Rearranging gives
\begin{equation} \label{eq: bound in A}
\norm{Pz-x}_2^2 = \norm{z-x}_2^2(1 - \inprod{\tilde{z},a}^2),
\end{equation}
where $\tilde{z} = (z-x)/\norm{z-x}_2$.

In the complement of this event, we get
\begin{equation*}
Pz = z + \inprod{a,(-x)-z}a = z - \inprod{a,z-x} + \inprod{a,-2x},
\end{equation*}
and using orthogonality,
\begin{equation} \label{eq: orthogonality for P}
\norm{Pz-x}_2^2 = \norm{z-x}_2^2 - \inprod{a,z-x}^2 + \inprod{a,2x}^2.
\end{equation}

\begin{figure}
	\begin{tabular}{c c}
		\begin{minipage}{0.45\textwidth}
			\begin{center}
				
				\begin{tikzpicture}[scale=1]
				% Draw points
				\draw[fill=black] (-1.5,0) circle[radius=2pt] node[below right] {$-x$};
				\draw[fill=black] (1.5,0) circle[radius=2pt] node[below right] {$x$};
				%	\draw[fill=black] (0,0) circle[radius=1pt] node[below left] {$0$};
				\draw[fill=blue] (0.8,1.5) circle[radius=2pt] node[below=2pt] {$z$};
				\draw[fill=blue] (1.9207,1.0517) circle[radius=2pt] node[below right] {$Pz$};
				
				% Draw lines
				\draw[thick]  (1.5-0.4,-1) node[below right] {$H_+$} -- (1.5+1,2.5) ;
				\draw[thick] (-1.5-0.4,-1) node[below left] {$H_-$} -- (-1.5+1,2.5);
				\draw[dashed,gray] (-0.4,-1) node[below] {$H_0$} -- (1,2.5);
				\draw[dotted] (0.8,1.5) -- (1.9207,1.0517);
				
				% Draw angle
				\draw[-]  (1.5+0.2350, 1.1260) -- (1.6607, 0.9403) -- (1.5+0.3464, 0.8660);
				
				\end{tikzpicture}
			\end{center}
		\end{minipage}
		&
		\begin{minipage}{0.45\textwidth}
			\begin{center}
				
				\begin{tikzpicture}[scale=1]
				% Draw points
				\draw[fill=black] (-1.5,0) circle[radius=2pt] node[below right] {$-x$};
				\draw[fill=black] (1.5,0) circle[radius=2pt] node[below right] {$x$};
				%	\draw[fill=black] (0,0) circle[radius=1pt] node[below left] {$0$};
				\draw[fill=blue] (0.8,1.5) circle[radius=2pt] node[left] {$z$};
				\draw[fill=blue] (-0.2294, 2.1176) circle[radius=2pt] node[above left] {$Pz$};
				%	\draw[fill=red] (1.9765, 0.7941) circle[radius=2pt] node[right] {$z - \inprod{a,z-x}$};	
				
				% Draw lines
				\draw[thick]  (1.5-0.6,-1) node[below right] {$H_+$} -- (1.5+1.5,2.5) ;
				\draw[thick] (-1.5-0.6,-1) node[below left] {$H_-$} -- (-1.5+1.5,2.5);
				\draw[dashed,gray] (-0.6,-1) node[below] {$H_0$} -- (1.5,2.5);
				\draw[dotted] (1.9765, 0.7941) -- (-0.2294, 2.1176);
				
				% Draw angle
				\draw[-]  (-0.3494, 1.9176) -- (-0.1779, 1.8147) -- (-0.0579, 2.0147);

				\draw[decoration={brace,mirror,raise=5pt},decorate,black]
				(1.9336, 0.8198) -- node[above right=5pt,black] {$|\inprod{a,2x}|$} (-0.1865, 2.0919);
				
				\draw[decoration={brace,raise=5pt},decorate,black]
				(1.9336, 0.8198) -- node[below left=5pt,black] {$|\inprod{a,z-x}|$} (0.8429, 1.4743);

				\end{tikzpicture}
			\end{center}
		\end{minipage}
	\end{tabular}
	\caption{Orientation of $x$, $z$, and $Pz$ when $a \in W_{x,z}$ and when $a \notin W_{x,z}$. $H_+$ and $H_-$ denote respectively the hyperplanes defined by the equations $\inprod{y,a} = b$ and $\inprod{y,a} = -b$. $H_0$ denotes the hyperplane defined by the equation $\inprod{y,a} = 0$. The left diagram demonstrates the situation when $a \in W_{x,z}$, thereby justifying \eqref{eq: bound in A pt 1}. The right diagram demonstrates the situation when $a \notin W_{x,z}$, thereby justifying \eqref{eq: bound in A^c}. }
	\label{fig: projection}
\end{figure}
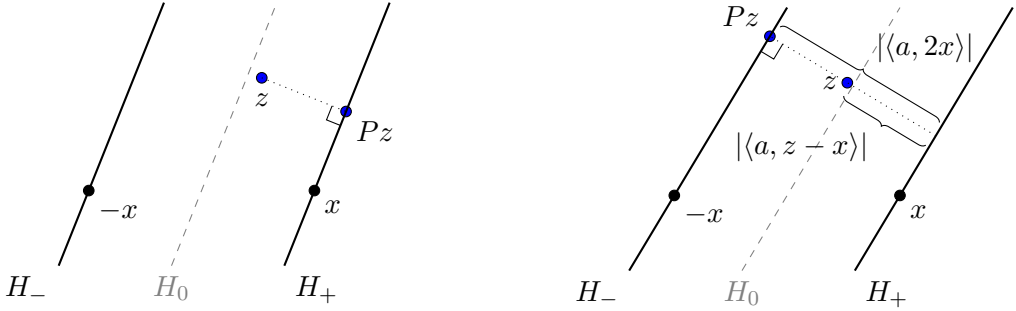

Since $z$ gets projected to the hyperplane containing $-x$, it may move further away from $x$. However, we can bound how far away it can move. Because $\inprod{a,x}$ has the opposite sign as $\inprod{a,z}$, we have
\begin{equation*}
\abs{\inprod{a,z+x}} < \abs{\inprod{a,z-x}},
\end{equation*}
and so
\begin{equation*}
\abs{\inprod{a,2x}} = \abs{\inprod{a,(z-x) - (z + x)}} < 2 \abs{\inprod{a,z-x}}.
\end{equation*}
Substituting this into \eqref{eq: orthogonality for P}, we get the bound
\begin{equation} \label{eq: bound in A^c}
\norm{Pz-x}_2^2 \leq \norm{z-x}_2^2 + 3\inprod{a,z-x}^2 = \norm{z-x}_2^2(1 + 3\inprod{\tilde{z},a}^2),
\end{equation}
where $\tilde{z}$ is as before.

We can combine \eqref{eq: bound in A} and \eqref{eq: bound in A^c} into a single inequality by writing
\begin{align*}
\norm{Pz-x}_2^2 & \leq \norm{z-x}_2^2(1 - \inprod{\tilde{z},a}^2)1_{W_{x,z}^c}(a) +  \norm{z-x}_2^2(1 + 3\inprod{\tilde{z},a}^2)1_{W_{x,z}}(a) \\
& = \norm{z-x}_2^2(1 - (1-4\cdot 1_{W_{x,z}}(a))\inprod{\tilde{z},a}^2) \\
& = \norm{z-x}_2^2(1 - \inprod{\tilde{z},(1-4\cdot 1_{W_{x,z}}(a))aa^T \tilde{z}}).
\end{align*}
Taking expectations, we can remove the role that $\tilde{z}$ plays by bounding this as follows.
\begin{align*}
\E\sqbracket{\norm{z-x}_2^2(1 - \inprod{\tilde{z},(1-4\cdot 1_{W_{x,z}}(a))aa^T \tilde{z}})} & = \norm{z-x}_2^2(1 - \inprod{\tilde{z},\E\sqbracket{(1-4\cdot 1_{W_{x,z}}(a))aa^T} \tilde{z}}) \\
& \leq \norm{z-x}_2^2\sqbracket*{1 -\lambda_{\min}(\E aa^T-4\E aa^T1_{W_{x,z}}(a))}.
\end{align*}

We may thus summarize what we have obtained in the following lemma.

\begin{lemma}[Expected decrement] \label{lem: expected decrement}
	Fix vectors $x,z \in \R^n$, a probability measure $\mu$ on $S^{n-1}$, and let $P = P_\mu$, $W_{x,z}$ be defined as in \eqref{eq: P definition} and \eqref{eq: W definition} respectively. Then
	\begin{equation*}
	\E\norm{Pz-x}_2^2 \leq \sqbracket*{1 -\lambda_{\min}(\E aa^T-4\E aa^T1_{W_{x,z}}(a))}\norm{z-x}_2^2.
	\end{equation*}
\end{lemma}

Let us next compute what happens for $\mu = \sigma$, the uniform measure on the sphere. It is easy to see that $\E aa^T = \frac{1}{n}I_n$, so it remains to compute $\E aa^T1_{W_{x,z}}(a)$. To do this, we make a convenient choice of coordinates: Let $\theta$ be the angle between $z$ and $x$. We assume that both points lie in the plane spanned by $e_1$ and $e_2$, the first two basis vectors, and that the angle between $z$ and $x$ is bisected by $e_1$, as illustrated in Figure \ref{fig: choice of coordinates}.

\begin{figure}[h]
	\includegraphics[scale=0.5]{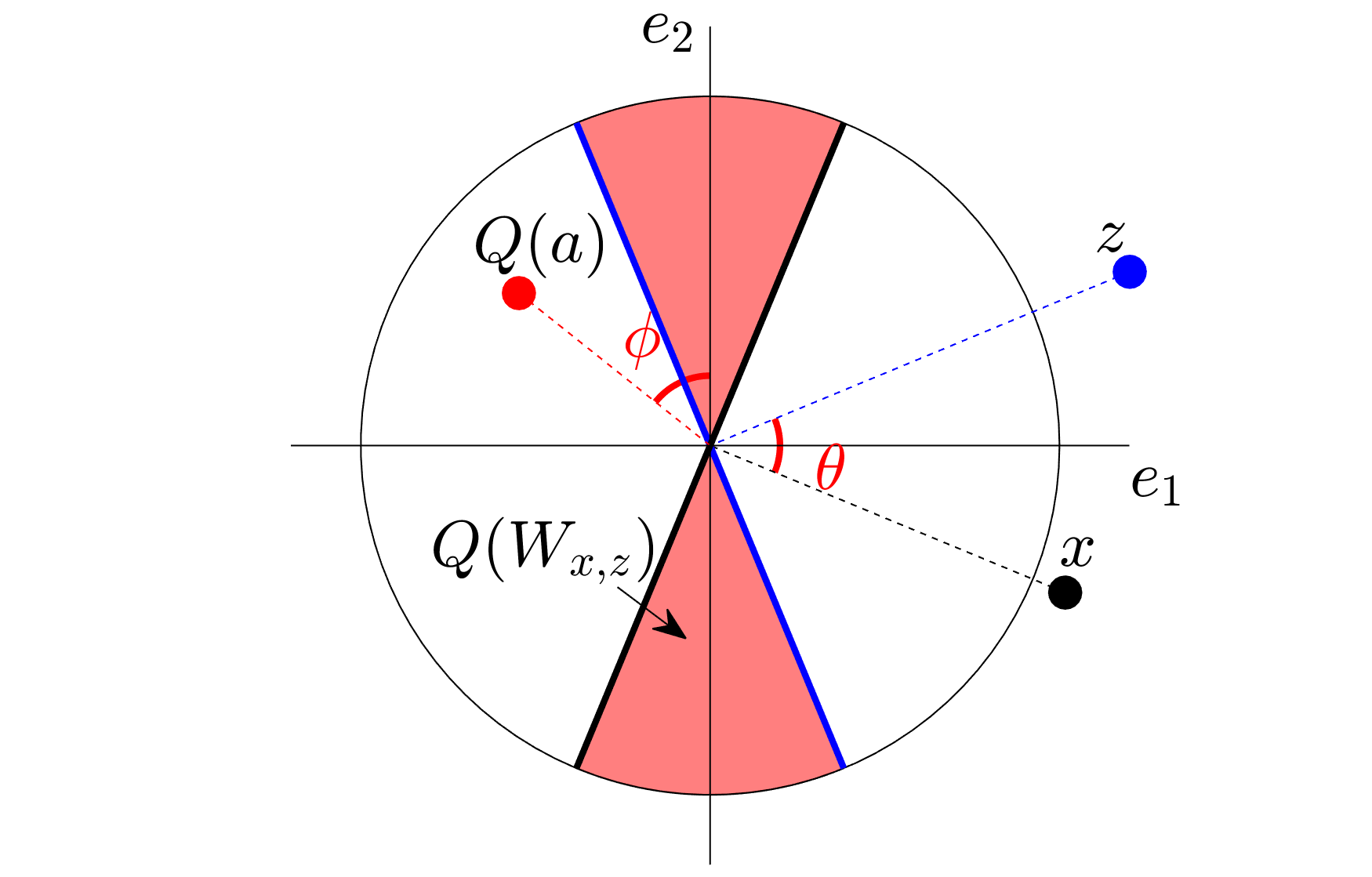}
	\centering
	\caption{Choice of coordinates}
	\label{fig: choice of coordinates}
\end{figure}

For convenience, denote $M := \E aa^T 1_{W_{x,z}}(a)$. Let $Q$ denote the orthogonal projection operator onto the span of $e_1$ and $e_2$. Then $Q(W_{x,z})$ is the union of two sectors of angle $\theta$, which are respectively bisected by $e_2$ and $-e_2$. Recall that all coordinate projections of the uniform random vector $a$ are uncorrelated. It is clear that from the symmetry in Figure \ref{fig: choice of coordinates} that they remain uncorrelated even when conditioning on the event that $a \in W_{x,z}$. As such, $M$ is a diagonal matrix.

Let $\phi$ denote the anti-clockwise angle of $Qa$ from $e_2$ (see Figure \ref{fig: choice of coordinates}). We may write
\begin{equation*}
\inprod{a,e_1}^2 = \norm{Qa}_2^2 \inprod{Qa/\norm{Qa}_2, e_1}^2 = \norm{Qa}_2^2 \sin^2\phi.
\end{equation*}

Note that the magnitude and direction of $Qa$ are independent, and $a \in W_{x,z}$ if either $\phi$ or $\phi - \pi$ lies between $-\theta/2$ and $\theta/2$. We therefore have
\begin{equation*}
M_{11} = \E \sqbracket{\inprod{a,e_1}^21_{W_{x,z}}(a)} = \E\sqbracket{\norm{Qa}_2^2\E\sin^2\phi 1_{(-\theta/2,\theta/2)}(\phi \text{ or }\phi-\pi)}.
\end{equation*}
By a standard calculation using symmetry, we have $\E\norm{Qa}_2^2 = 2/n$. Since $\phi$ is distributed uniformly on the circle, we can compute
\begin{equation*}
\E\sin^2\phi 1_{(-\theta/2,\theta/2)}(\phi \text{ or }\phi-\pi) = \frac{1}{\pi}\int_{-\theta/2}^{\theta/2} \sin^2 t dt = \frac{1}{\pi} \int_{-\theta/2}^{\theta/2} \frac{1-\cos(2t)}{2} dt = \frac{\theta-\sin\theta}{2\pi}.
\end{equation*}

As such, we have $M_{11} = (\theta - \sin\theta)/n\pi$, and by a similar calculation, $M_{22} = (\theta + \sin\theta)/n\pi$. Meanwhile, for $i \geq 3$ we have
\begin{align*}
M_{ii} & = \frac{\Tr\paren{M} - M_{11} - M_{22}}{n-2} \\
& =  \frac{\E\sqbracket{\norm{(I-Q)a}_2^21_{W_{x,z}}(a)}}{n-2} \\
& = \frac{\E\norm{(I-Q)a}_2^2\E 1_{(-\theta/2,\theta/2)}(\phi \text{ or }\phi-\pi)}{n-2} \\
& = \frac{(n-2)/n \cdot \theta/\pi }{n} = \frac{\theta}{n\pi}.
\end{align*}
This implies that
\begin{equation} \label{eq: max eigenvalue of covariance conditioned on wedge}
\lambda_{\max}(M_\theta) = \frac{\theta + \sin\theta}{n\pi}.
\end{equation}
We have now completed proving the following lemma.

\begin{lemma}[Expected decrement for uniform measure] \label{lem: expected decrement, uniform measure}
	Fix vectors $x, z \in \R^n$ such that $\inprod{z,x} > 0$, and let $P = P_\sigma$ denote the generalized Kaczmarz projection with respect to $\sigma$, the uniform measure on the sphere. Let $\theta$ be the angle between $z$ and $x$. Then
	\begin{equation*}
	\E\norm{Pz-x}_2^2 \leq \Big[ 1 - \frac{1-4(\theta+\sin\theta)/\pi}{n}   \Big] \norm{z-x}_2^2.
	\end{equation*}
\end{lemma}

\begin{remark}
	By being more careful, one may compute an \emph{exact} formula for the expected decrement rather than a bound as is the case in previous lemma. This is not necessary for our purposes and does not give better guarantees in our analysis, so the computation is omitted.
\end{remark}

\section{Local linear convergence using unlimited uniform measurements}

In this section, we will show that if we start with an initial estimate that is close enough to the ground truth $x$, then repeatedly applying generalized Kaczmarz projections with respect to the uniform measure $\sigma$ gives linear convergence in expectation. This is exactly the situation we would be in if we were to run randomized Kaczmarz given an unlimited supply of independent sampling vector $a_1,a_2,\ldots$ drawn uniformly from the sphere.

We would like to imitate the proof for linear convergence of randomized Kaczmarz for linear systems (Theorem \ref{thm: linear convergence for linear systems}) given in \cite{Strohmer2009}. We denote by $X_k$ the estimate after $k$ steps, using capital letters to emphasize the fact that it is a random variable. If we know that $X_k$ takes the value $x_k \in \R^n$, and the angle $\theta_k$ that $z$ makes with $x_k$ is smaller than $\pi/8$, then, Lemma \ref{lem: expected decrement, uniform measure} tells us

\begin{equation} \label{eq: decrement conditioned on X_k}
\E\sqbracket{\norm{X_{k+1}-x}_2^2 \st X_k = x_k} \leq (1-\alpha_\sigma/n)\norm{x_k-x}_2^2,
\end{equation}
where $\alpha_\sigma := 1/2 - 4\sin(\pi/8)/\pi > 0$.

The proof for Theorem \ref{thm: linear convergence for linear systems} proceeds by unconditioning and iterating a bound similar to \eqref{eq: decrement conditioned on X_k}. Unfortunately, our bound depends on $x_k$ being in a specific region in $\R^n$ and does not hold arbitrarily. Nonetheless, by using some basic concepts from stochastic process theory, we may derive a \emph{conditional} linear convergence estimate. The details are as follows.

For each $k$, let $\mathcal{F}_k$ denote the $\sigma$-algebra generated by $a_1,a_2,\ldots,a_k$, where $a_k$ is the sampling vector used in step $k$. Let $B \subset \R^n$ be the region comprising all points making an angle less than or equal to $\pi/8$ with $x$. This is our \emph{basin of linear convergence}. Let us assume a fixed initial estimate $x_0 \in B$. Now define a stopping time $\tau$ via

\begin{equation} \label{eq: definition of tau}
\tau := \min \braces{k ~\colon ~ X_k \notin B}.
\end{equation}

For each $k$, and $x_k \in B$, we have

\begin{align*}
\E\sqbracket{\norm{X_{k+1}-x}_2^21_{\tau > {k+1}} \st X_k = x_k} & \leq  \E\sqbracket{\norm{X_{k+1}-x}_2^2 1_{\tau > k} \st X_k = x_k}  \\
& = \E\sqbracket{\norm{X_{k+1}-x}_2^21_{\tau > k} \st X_k = x_k, \mathcal{F}_k} \\
& = \E\sqbracket{\norm{X_{k+1}-x}_2^2 \st X_k = x_k, \mathcal{F}_k}1_{\tau > k} \\
& \leq (1-\alpha_\sigma/n)\norm{x_k-x}_2^2 1_{\tau > k}.
\end{align*}

Here, the first inequality follows from the inclusion $\braces{\tau > {k+1}} \subset \braces{\tau > k}$, the first equality statement from the Markov nature of the process $(X_k)$, the second equality statement from the fact that $\tau$ is a stopping time, while the second inequality is simply \eqref{eq: decrement conditioned on X_k}. Taking expectations with respect to $X_k$ then gives

\begin{align*}
\E\sqbracket{\norm{X_{k+1}-x}_2^2 1_{\tau > {k+1}}} & = \E\sqbracket{\E\sqbracket{\norm{X_{k+1}-x}_2^21_{\tau > {k+1}} \st X_k} } \\
& \leq (1-\alpha_\sigma/n)\E\sqbracket{\norm{X_k-x}_2^2 1_{\tau > k}}.
\end{align*}

By induction, we therefore obtain
\begin{equation*}
\E\sqbracket{\norm{X_k-x}_2^2 1_{\tau > k}} \leq (1-\alpha_\sigma/n)^k\norm{x_0-x}_2^2.
\end{equation*}

We have thus proven the first part of the following convergence theorem.

\begin{theorem}[Linear convergence from unlimited measurements] \label{thm: linear convergence, unlimited measurements}
	Let $x$ be a vector in $\R^n$, let $\delta > 0$, and let $x_0$ be an initial estimate to $x$ such that $\norm{x_0 - x}_2 \leq \delta\norm{x}_2$. Suppose that our measurements $a_1,a_2,\ldots$ are fully independent random vectors distributed uniformly on the sphere $S^{n-1}$. Let $X_k$ be the estimate given by the randomized Kaczmarz update formula \eqref{def: generalized Kaczmarz} at step $k$, and let $\tau$ be the stopping time defined via \eqref{eq: definition of tau}. Then for every $k \in \Z_+$,
	\begin{equation}
	\E\sqbracket{\norm{X_k-x}_2^2 1_{\tau = \infty}} \leq (1-\alpha_\sigma/n)^k\norm{x_0-x}_2^2,
	\end{equation}
	where $\alpha_\sigma = 1/2 - 4\sin(\pi/8)/\pi > 0$. Furthermore, $\P(\tau < \infty) \leq (\delta/\sin(\pi/8))^2$.
\end{theorem}

\begin{proof}
	In order to prove the second statement, we combine a stopping time argument with a supermartingale maximal inequality.	Set $Y_k := \norm{X_{\tau \wedge k}-x}_2^2$. We claim that $Y_k$ is a supermartingale. To see this, we break up its conditional expectation as follows: 
	\begin{align*}
	\E\sqbracket{Y_{k+1} \st \mathcal{F}_k} & = \E\sqbracket{\norm{X_{\tau \wedge (k+1)}-x}_2^21_{\tau \leq k} \st \mathcal{F}_k} + \E\sqbracket{\norm{X_{\tau \wedge (k+1)}-x}_2^21_{\tau > k} \st \mathcal{F}_k} \\
	& = \E\sqbracket{\norm{X_{\tau \wedge k}-x}_2^21_{\tau \leq k} \st \mathcal{F}_k} + \E\sqbracket{\norm{X_{k+1}-x}_2^21_{\tau > k} \st \mathcal{F}_k}.
	\end{align*}
	
	Since $\norm{X_{\tau \wedge k}-x}_2^2$ is measurable with respect to $\mathcal{F}_k$, we get
	\begin{equation*}
	\E\sqbracket{\norm{X_{\tau \wedge k}-x}_2^21_{\tau \leq k} \st \mathcal{F}_k} = \norm{X_{\tau \wedge k}-x}_2^21_{\tau \leq k} = Y_k 1_{\tau \leq k}.
	\end{equation*}
	Meanwhile, on the event $\tau > k$, we have $X_k \in B$, so we may use \eqref{eq: decrement conditioned on X_k} to obtain
	\begin{equation*}
	\E\sqbracket{\norm{X_{k+1}-x}_2^21_{\tau > k} \st \mathcal{F}_k} = \E\sqbracket{\norm{X_{k+1}-x}_2^2 \st \mathcal{F}_k}1_{\tau > k} \leq (1-\alpha_\sigma/n)\norm{X_k-x}_2^21_{\tau > k}.
	\end{equation*}
	Next, notice that
	\begin{equation*}
	\norm{X_k-x}_2^21_{\tau > k} = \norm{X_{\tau \wedge k}-x}_2^21_{\tau > k} = Y_k 1_{\tau > k}.
	\end{equation*}
	Combining these calculations gives
	\begin{equation*}
	\E\sqbracket{Y_{k+1} \st \mathcal{F}_k} \leq Y_k 1_{\tau \leq k} + (1-\alpha_\sigma/n)Y_k 1_{\tau > k} \leq Y_k.
	\end{equation*}
	
	Now define a second stopping time $T$ to be the earliest time $k$ such that $\norm{X_k-x}_2 \geq \sin(\pi/8) \cdot \norm{x}_2$. A simple geometric argument tells us that $T \leq \tau$, and that $T$ also satisfies
	\begin{equation*}
	T = \inf\braces{k \st Y_k \geq \sin^2(\pi/8)\norm{x}_2^2}.
	\end{equation*}
	As such, we have
	\begin{equation*}
	\P(\tau < \infty ) \leq \P(T < \infty) = \P\Big(\sup_{1 \leq k < \infty} Y_k \geq \sin^2(\pi/8) \norm{x}_2^2\Big).
	\end{equation*}
	Since $(Y_k)$ is a non-negative supermartingale, we may apply the supermartingale maximal inequality to obtain a bound on the right hand side:
	\begin{equation*}
	\P\Big(\sup_{1 \leq k < \infty} Y_k \geq \sin^2(\pi/8) \norm{x}_2^2\Big) \leq \frac{\E Y_0}{\sin^2(\pi/8)\norm{x}_2^2} \leq (\delta/\sin(\pi/8))^2.
	\end{equation*}
	This completes the proof of the theorem.
\end{proof}

\begin{corollary}
	Fix $\epsilon > 0$, $0 < \delta_1 \leq 1/2$, and $0 < \delta_2 \leq 1$. In the setting of Theorem \ref{thm: linear convergence, unlimited measurements}, suppose that $\norm{x_0-x}_2 \leq \sqrt{\delta_1}\sin(\pi/8)\norm{x}_2$. Then with probability at least $1-\delta_1-\delta_2$, if $k \geq (\log(2/\epsilon)+\log(1/\delta_2))n/\alpha_\sigma$ then $\norm{X_k-x}_2^2 \leq \epsilon\norm{x_0-x}_2^2$.
\end{corollary}

\begin{proof}
	First observe that
	\begin{equation*}
	\P(\tau < \infty) \leq \paren*{\frac{\sqrt{\delta_1}\sin(\pi/8)}{\sin(\pi/8)}}^2 = \delta_1 \leq 1/2. 
	\end{equation*}
	
	Next, since
	\begin{align*}
	\E\sqbracket{\norm{X_k-x}_2^2 1_{\tau = \infty}} & = \E\sqbracket{\norm{X_k-x}_2^2 \st \tau = \infty}\P(\tau = \infty) + 0 \cdot \P(\tau < \infty) \\
	& \geq \frac{1}{2}\E\sqbracket{\norm{X_k-x}_2^2 \st \tau = \infty},
	\end{align*}
	applying Theorem \ref{thm: linear convergence, unlimited measurements} gives
	\begin{align*}
	\E\sqbracket{\norm{X_k-x}_2^2 \st \tau = \infty} \leq 2(1-\alpha_\sigma/n)^k\norm{x_0-x}_2^2.
	\end{align*}
	Applying Markov's inequality then gives
	\begin{align*}
	\P\paren{\norm{X_k-x}_2^2 > \epsilon \norm{x_0-x}_2^2 \st \tau = \infty} & \leq \frac{\E\sqbracket{\norm{X_k-x}_2^2 \st \tau = \infty}}{\epsilon \norm{x_0-x}_2^2} \\
	& \leq \frac{2(1-\alpha_\sigma/n)^k}{\epsilon}.
	\end{align*}
	
	Plugging our choice of $k$ into this last bound shows that it is in turn bounded by $\delta_2$. We therefore have
	\begin{align*}
	\P\paren{\norm{X_k-x}_2^2 \leq \epsilon \norm{x_0-x}_2^2 } & = \P\paren{\norm{X_k-x}_2^2 \leq \epsilon \norm{x_0-x}_2^2 \st \tau = \infty} \P\paren{\tau = \infty} \\
	& \geq (1-\delta_2)(1-\delta_1) \\
	& \geq 1 - \delta_1 - \delta_2
	\end{align*}
	as we wanted.
\end{proof}

\section{Local linear convergence for $\ACW(\theta,\alpha)$ measures}

We would like to extend the analysis in the previous section to the setting where we only have access to finitely many uniform measurements, i.e. when we are back in the situation of \eqref{eq: phase retrieval}. When we sample uniformly from the rows of $A$, this can be seen as running the generalized randomized Kaczmarz algorithm using the measure $\mu_A = \frac{1}{m}\sum_{i=1}^m \delta_{a_i}$ as opposed to $\mu = \sigma$.

If we retrace our steps, we will see that the key property of the uniform measure $\sigma$ that we used was that if $W \subset S^{n-1}$ is a wedge\footnote{Recall that a wedge of angle $\theta$ is the region of the sphere between two hemispheres with normal vectors making an angle of $\theta$.} of angle $\theta$, then we could make $\lambda_{\max}(\E_\sigma aa^T 1_{W}(a))$ arbitrarily small by taking $\theta$ small enough (see equation \eqref{eq: max eigenvalue of covariance conditioned on wedge}). We do not actually need such a strong statement. It suffices for there to be an absolute constant $\alpha$ such that
\begin{equation} \label{eq: max eigenvalue goal}
\lambda_{\min}(\E aa^T-4\E aa^T1_{W}(a)) \geq \frac{\alpha}{n}
\end{equation}
holds for $\theta$ small enough.

\begin{definition}[Anti-concentration]
  If a probability measure $\mu$ on $S^{n-1}$ satisfies $\eqref{eq: max eigenvalue goal}$ for all wedges $W$ of angle less than $\theta$, we say that it is \emph{anti-concentrated on wedges of angle $\theta$ at level $\alpha$}, or for short, that it satisfies the $\ACW(\theta,\alpha)$ condition. 
\end{definition}

Abusing notation, we say that a measurement matrix $A$ is $\ACW(\theta,\alpha)$ if the uniform measure on its rows is $\ACW(\theta,\alpha)$. Plugging in this definition into Lemma \ref{lem: expected decrement}, we immediately get the following statement.

\begin{lemma}[Expected decrement for $\ACW$ measure] \label{lem: expected decrement, ACW measure}
	Let $\mu$ be a probability measure on the sphere $S^{n-1}$ satisfying the $ACW(\theta,\alpha)$ condition for some $\alpha > 0$ and some acute angle $\theta > 0$. Let $P = P_\mu$ denote the generalized Kaczmarz projection with respect to $\mu$. Then for any $x, z \in \R^n$ such that the angle between them is less than $\theta$, we have
	\begin{equation}
	\E\norm{Pz-x}_2^2 \leq (1-\alpha/n) \norm{z-x}_2^2.
	\end{equation}
\end{lemma}

We may now imitate the arguments in the previous section to obtain a guarantee for local linear convergence for the generalized randomized Kaczmarz algorithm using such a measure $\mu$.

\begin{theorem}[Linear convergence for $\ACW$ measure] \label{thm: linear convergence, ACW measure}
	Suppose $\mu$ is an $\ACW(\theta,\alpha)$ measure. Let $x$ be a vector in $\R^n$, let $\delta > 0$, and let $x_0$ be an initial estimate to $x$ such that $\norm{x_0 - x}_2 \leq \delta\norm{x}_2$. Let $X_k$ denote the $k$-th step of the generalized randomized Kaczmarz method with respect to the measure $\mu$, defined as in \eqref{def: generalized Kaczmarz}. Let $\Omega$ be the event that for every $k \in \Z_+$, $X_k$ makes an angle less than $\theta$ with $x$. Then for every $k \in \Z_+$,
	\begin{equation} \label{eq: linear convergence for ACW measure}
	\E\sqbracket{\norm{X_k-x}_2^2 1_\Omega} \leq (1-\alpha/n)^k\norm{x_0-x}_2^2.
	\end{equation}
	Furthermore, $\P(\Omega^c) \leq (\delta/\sin\theta)^2$. 
\end{theorem}

\begin{proof}
	We repeat the proof of Theorem \ref{thm: linear convergence, unlimited measurements}. Let $B_\mu \subset S^{n-1}$ be the region on the sphere comprising all points making an angle less than or equal to $\pi/8$ with $x$. Define stopping times $\tau_\mu$ and $T_\mu$ as the earliest times that $X_k \notin B_\mu$ and $\norm{X_k - x_0}_2 \geq \sin(\theta)\norm{x}_2$ respectively. Again, $Y_k := X_{k\wedge\tau_\mu}$ is a supermartingale, so we may use the supermartingale inequality to bound the probability of $\Omega^c$. Conditioned on the event $\Omega$, we may iterate the bound given by Lemma \ref{lem: expected decrement, ACW measure} to obtain \eqref{eq: linear convergence for ACW measure}.
\end{proof}

\begin{corollary} \label{cor: linear convergence for ACW measure, high probability}
	Fix $\epsilon > 0$, $0 < \delta_1 \leq 1/2$, and $0 < \delta_2 \leq 1$. In the setting of Theorem \ref{thm: linear convergence, ACW measure}, suppose that $\norm{x_0-x}_2 \leq \sqrt{\delta_1}\sin(\theta)\norm{x}_2$. Then with probability at least $1-\delta_1-\delta_2$, if $k \geq (\log(2/\epsilon)+\log(1/\delta_2))n/\alpha$ then $\norm{X_k-x}_2^2 \leq \epsilon\norm{x_0-x}_2^2$.
\end{corollary}

\section{$\ACW(\theta,\alpha)$ condition for finitely many uniform measurements} \label{sec: ACW condition}

Following the theory in the previous section, we see that to prove linear convergence from finitely many uniform measurements, it suffices to show that the measurement matrix $A$ is $\ACW(\theta,\alpha)$ for some $\theta$ and $\alpha$.

For a \emph{fixed} wedge $W$, we can easily achieve \eqref{eq: max eigenvalue goal} by using a standard matrix concentration theorem. By taking a union bound, we can guarantee that it holds over exponentially many wedges with high probability. However, the function $W \mapsto \lambda_{\max}(\E aa^T 1_{W}(a))$ is not Lipschitz with respect to any natural parametrization of wedges in $S^{n-1}$, so a naive net argument fails. To get around this, we use VC theory, metric entropy, and a chaining theorem from \cite{Dirksen2015}.

First, we will use the theory of VC dimension and growth functions to argue that all wedges contain approximately the right fraction of points. This is the content of lemma \ref{lem: uniform conc over wedges}. In order to prove this, a fair number of standard definitions and results are required. These are all provided in Appendix \ref{sec: growth fns and VC dim}.

\begin{lemma}[Uniform concentration of empirical measure over wedges] \label{lem: uniform conc over wedges}
	Fix an acute angle $\theta > 0$. Let $\mathcal{W}_\theta$ denote the collection of all wedges of $S^{n-1}$ of angle less than $\theta$. Suppose $A$ is an $m$ by $n$ matrix with rows $a_i$ that are independent uniform random vectors on $S^{n-1}$, and let $\mu_A = \frac{1}{m}\sum_{i=1}^m \delta_{a_i}$. Then if $m \geq (4\pi/\theta)^2(2n\log(2em/n)+\log(2/\delta))$, with probability at least $1 - \delta$, we have
	\begin{equation*}
	\sup_{W \in \mathcal{W}} \mu_A(W) \leq 2\theta/\pi.
	\end{equation*}
\end{lemma}

\begin{proof}
	Using VC theory \cite{Vapnik1971a}, we have
	\begin{equation} \label{eq: VC union bound}
	\P\paren{\sup_{W \in \mathcal{W}} \abs{\mu_A(W) - \sigma(W)} \geq u } \leq 4\Pi_{\mathcal{W}_\theta}(2m)\exp(-mu^2/16)
	\end{equation}
	whenever $m \geq 2/u^2$. Let $\mathcal{S}$ be the collection of all sectors of any angle, and let $\mathcal{H}$ denote the collection of all hemispheres. By Claim \ref{clm: VC dim for hemispheres} and the Sauer-Shelah lemma (Lemma \ref{lem: sauer-shelah}) relating VC dimension to growth functions, we have $\Pi_\mathcal{H}(2m) \leq (2em/n)^n$.
	
	Next, notice that using the notation in \eqref{def: symmetric difference}, we have $\mathcal{W} = \mathcal{H}\triangle\mathcal{H}$. As such, we may apply Claim \ref{clm: growth function bound for symmetric difference} to get
	\begin{equation*}
	\Pi_\mathcal{W}(2m) \leq (2em/n)^{2n}.
	\end{equation*}
	
	We now plug this bound into the right hand side of \eqref{eq: VC union bound}, set $u = \theta/\pi$, and simplify to get
	\begin{equation*}
	\P\paren{\sup_{W \in \mathcal{W}} \abs{\mu_A(W) - \sigma(W)} \geq \theta/\pi } \leq 4\exp(2n\log(2em/n)-m(\theta/\pi)^2/16).
	\end{equation*}
	
	Our assumption implies that $m \geq 2/(\theta/\pi)^2$ so the bound holds, and also that the bound is less than $\delta$. Finally, since $\mathcal{W}_\theta \subset \mathcal{W}$, on the complement of this event, any $W \in \mathcal{W}_\theta$ satisfies
	\begin{equation*}
	\mu_A(W) \leq \sigma(W) + \theta/\pi \leq 2\theta/\pi
	\end{equation*}
	as we wanted.
\end{proof}

For every wedge $W \in \mathcal{W}_\theta$, we may associate the configuration vector
\begin{equation*}
s_{W,A} := \paren{1_W(a_1),1_W(a_2),\ldots,1_W(a_m)}.
\end{equation*}
We can write
\begin{equation} \label{eq: max eigenvalue with weights}
\lambda_{\max}\paren{\E_{\mu_A}aa^T1_W(a)} = \frac{1}{m}\lambda_{\max}\paren{A^TS_{W,A}A},
\end{equation}
where $S_{W,A} = \diag(s_{W,A})$. $S_{W,A}$ is thus a selector matrix, and if we condition on the good event given to us by the previous theorem, it selects at most a $2\theta/\pi$ fraction of the rows of $A$. This means that $s_{W,A} \in \mathcal{S}_{2\theta/\pi}$, where we define
\begin{equation*}
\mathcal{S}_\tau := \braces{d \in \braces{0,1}^m \st \inprod{d,1} \leq \tau \cdot m }.
\end{equation*}

We would like to majorize the quantity in \eqref{eq: max eigenvalue with weights} uniformly over all wedges $W$ by the quantity $\frac{1}{4}\lambda_{\min}(\E_{\mu_A} aa^T)$. In order to do this, we define a stochastic process $(Y_{s,v})$ indexed by $s \in \mathcal{S}_{2\theta/\pi}$ and $v \in B^n_2$, setting
\begin{equation} \label{def: Y process}
Y_{s,v} := n v^T A^T\diag(s)A = \sum_{i=1}^m s_i \inprod{\sqrt{n}a_i,v}^2.
\end{equation}
If we condition on the good set in Lemma \ref{lem: uniform conc over wedges}, it is clear that
\begin{equation*}
\sup_{W \in \mathcal{W}_\theta} \frac{1}{m}\lambda_{\max}\paren{A^TS_{W,A}A} \leq \frac{1}{nm}\sup_{s \in \mathcal{S}_{2\theta/\pi},v \in B^n_2} Y_{s,v},
\end{equation*}
so it suffices to bound the quantity on the right. We will do this using a slightly sophisticated form of chaining, which requires us to make a few definitions.

Let $(T,d)$ be a metric space. A sequence $\mathcal{T} = (T_k)_{k \in \Z_+}$ of subsets of $T$ is called \emph{admissible} if $\abs{T_0} = 1$, and $\abs{T_k} \leq 2^{2^k}$ for all $k \geq 1$. For any $0 < \alpha < \infty$, we define the \emph{$\gamma_\alpha$ functional} of $(T,d)$ to be
\begin{equation*}
\gamma_\alpha(T,d) := \inf_{\mathcal{T}}\sup_{t \in T} \sum_{k=0}^\infty 2^{k/\alpha} d(t,T_k).
\end{equation*}

Let $d_1$ and $d_2$ be two metrics on $T$. We say that a process $(Y_t)$ has \emph{mixed tail increments} with respect to $(d_1,d_2)$ if there are constants $c$ and $C$ such that for all $s, t \in T$, we have the bound
\begin{equation} \label{def: mixed tail increments}
\P\paren{\abs{Y_s-Y_t} \geq c(\sqrt{u}d_2(s,t) + ud_1(s,t)) } \leq Ce^{-u}.
\end{equation}

\begin{remark}
	In \cite{Dirksen2015}, processes with mixed tail increments are defined as above but with the further restriction that $c=1$ and $C=2$. This is not necessary for the result that we need (Lemma \ref{lem: sup bound for mixed tail process}) to hold. The indeterminacy of $c$ and $C$ gets absorbed into the final constant in the bound.
\end{remark}

\begin{lemma}[Theorem 5, \cite{Dirksen2015}] \label{lem: sup bound for mixed tail process}
	If $(Y_t)_{t \in T}$ has mixed tail increments, then there is a constant $C$ such that for any $u \geq 1$, with probability at least $1 - e^{-u}$,
	\begin{equation*}
	\sup_{t \in T}\abs{Y_t - Y_{t_0}} \leq C\paren{\gamma_2(T,d_2) + \gamma_1(T,d_1) + \sqrt{u}\diam(T,d_2) + u\diam(T,d_1)}.
	\end{equation*}
\end{lemma}

At first glance, the $\gamma_2$ and $\gamma_1$ quantities seem mysterious and intractable. We will show however, that they can be bounded by more familiar quantities that are easily computable in our situation. Let us postpone this for the moment, and first show that our process $(Y_{s,v})$ has mixed tail increments.

\begin{lemma}[$(Y_{s,v})$ has mixed tail increments]
	Let $(Y_{s,v})$ be the process defined in \eqref{def: Y process}. Define the metrics $d_1$ and $d_2$ on $\mathcal{S}_{2\theta/\pi} \cross B^n_2$ using the norms $\tripnorm{(w,v)}_1 = \max\braces{\norm{w}_\infty,\norm{v}_2}$ and $\tripnorm{(w,v)}_2 = \max\braces{\norm{w}_2,\sqrt{2m\theta/\pi}\norm{v}_2}$. Then the process has mixed tail increments with respect to $(d_1,d_2)$.
\end{lemma}

\begin{proof}
	The main tool that we use is Bernstein's inequality \cite{Vershynin2011a} for sums of subexponential random variables. Observe that each $\sqrt{n}a_i$ is a subgaussian random vector with bounded subgaussian norm $\norm{\sqrt{n}a_i}_{\psi_2} \leq C$, where $C$ by an absolute constant. As such, for any $v \in B^n_2$, $\inprod{\sqrt{n}a_i,v}^2$ is a subexponential random variable with bounded subexponential norm $\norm{\inprod{\sqrt{n}a_i,v}^2}_{\psi_1} \leq C^2$ \cite{Vershynin2011a}.
	
	Now fix $v$ and let $s, s' \in \mathcal{S}_{2\theta/\pi}$. Then
	\begin{equation*}
	Y_{s,v} - Y_{s',v} = \sum_{i=1}^m (s_i- s_i') \inprod{\sqrt{n}a_i,v}^2.
	\end{equation*}
	Using Bernstein, we have
	\begin{equation} \label{eq: Bernstein for s}
	\P\paren{\abs{Y_{s,v} - Y_{s',v}} \geq u } \leq 2\exp(-c\min\braces{u^2/\norm{s-s'}_2^2,u/\norm{s-s'}_\infty}).
	\end{equation}
	
	Similarly, if we fix $s \in \mathcal{S}_{2\theta/\pi}$ and let $v, v' \in B^n_2$, then
	\begin{align*}
	Y_{s,v} - Y_{s,v'} & = \sum_{i=1}^m s_i (\inprod{\sqrt{n}a_i,v}^2 - \inprod{\sqrt{n}a_i,v'}^2 ) \\
	& = \sum_{i=1}^m s_i \inprod{\sqrt{n}a_i,v-v'}\inprod{\sqrt{n}a_i,v+v'}.
	\end{align*}
	We can bound the subexponential norm of each summand via
	\begin{align*}
	\norm{s_i\inprod{\sqrt{n}a_i,v-v'}\inprod{\sqrt{n}a_i,v+v'}}_{\psi_1} & \leq 	s_i\norm{\inprod{\sqrt{n}a_i,v-v'}}_{\psi_2} \cdot 	\norm{\inprod{\sqrt{n}a_i,v+v'}}_{\psi_1} \\
	& \leq Cs_i\norm{v-v'}_2.
	\end{align*}
	As such,
	\begin{equation*}
	\sum_{i=1}^m \norm{s_i\inprod{\sqrt{n}a_i,v-v'}\inprod{\sqrt{n}a_i,v+v'}}_{\psi_1}^2 \leq C \norm{v-v'}_2^2\sum_{i=1}^m s_i^2 \leq C(2\theta/\pi)m \norm{v-v'}_2^2.
	\end{equation*}
	Applying Bernstein as before, we get
	\begin{equation} \label{eq: Bernstein for v}
	\P\paren{\abs{Y_{s,v} - Y_{s,v'}} \geq u } \leq 2\exp(-c\min\braces{u^2/ (2\theta/\pi)m\norm{v-v'}_2^2,u/\norm{v-v'}_2}).
	\end{equation}

	Now, recall the simple observation that for any numbers $a, b \in \R$, we have
	\begin{equation*}
	\max\braces{\abs{a},\abs{b}} \leq \abs{a} + \abs{b} \leq 2\max\braces{\abs{a},\abs{b}}.
	\end{equation*}
	As such, for any $u > 0$, given $s,s' \in \mathcal{S}_{2\theta/\pi}$, $v,v' \in B^n_2$, we have
	\begin{align*}
	\sqrt{u}\tripnorm{(s,v) - (s',v')}_2 + u\tripnorm{(s,v) - (s',v')}_1 & \geq \frac{1}{2}(\sqrt{u}\norm{s-s'}_2 + \sqrt{u}\sqrt{2m\theta/\pi}\norm{v-v'}_2 + u\norm{s-s'}_\infty \nonumber \\
	& \quad \quad + u\norm{v-v'}_2) \\
	& \geq \frac{1}{2}\max\{\sqrt{u}\norm{s-s'}_2 + u\norm{s-s'}_\infty, \nonumber \\
	& \quad\quad \sqrt{u}\sqrt{2m\theta/\pi}\norm{s-s'}_2 + u\norm{v-v'}_2 \}.
	\end{align*}
	
	Since
	\begin{equation*}
	\abs{Y_{s,v} - Y_{s',v'}} \leq \abs{Y_{s,v} - Y_{s',v}} + \abs{Y_{s',v} - Y_{s',v'}},
	\end{equation*}
	we have that if
	\begin{equation*}
	\abs{Y_{s,v} - Y_{s',v'}} \geq c\paren*{\sqrt{u}\tripnorm{(s,v) - (s',v')}_2 + u\tripnorm{(s,v) - (s',v')}_1},
	\end{equation*}
	then either
	\begin{equation*}
	\abs{Y_{s,v} - Y_{s',v}} \geq \frac{c}{4}\paren*{\sqrt{u}\norm{s-s'}_2 + u\norm{s-s'}_\infty}
	\end{equation*}
	or
	\begin{equation*}
	\abs{Y_{s',v} - Y_{s',v'}} \geq \frac{c}{4}\paren*{\sqrt{u}\sqrt{2m\theta/\pi}\norm{v-v'}_2 + u\norm{v-v'}_2 }.
	\end{equation*}
	
	We can then combine the bounds \eqref{eq: Bernstein for v} and \eqref{eq: Bernstein for s} to get
	\begin{align*}
	\P\paren*{\abs{Y_{s,v} - Y_{s',v'}} \geq c\paren*{\sqrt{u}\tripnorm{(w,v) - (w',v')}_2 + u\tripnorm{(w,v) - (w',v')}_1}} \leq 4e^{-u}.
	\end{align*}
	Hence, the process $(Y_{s,v})$ satisfies the definition \eqref{def: mixed tail increments} for having mixed tail increments.
\end{proof}

We next bound the $\gamma_1$ and $\gamma_2$ functions for $\mathcal{S}_{2\theta/\pi} \cross B^n_2$.

\begin{lemma}
	We may bound the $\gamma_1$ functional of $\mathcal{S}_{2\theta/\pi} \cross B^n_2$ by
	\begin{equation*}
	\gamma_1(\mathcal{S}_{2\theta/\pi} \cross B^n_2,\tripnorm{\cdot}_1) \leq C\paren{(2\theta/\pi)\log(\pi/2\theta)m + n}.
	\end{equation*}
\end{lemma}

\begin{proof}
	The proof of the bound uses metric entropy and a version of Dudley's inequality. Let $(T,d)$ be a metric space, and for any $u > 0$, let $N(T,d,u)$ denote the covering number of $T$ at scale $u$, i.e. the smallest number of radius $u$ balls needed to cover $T$. Dudley's inequality (see \cite{talagrand2005generic}) states that there is an absolute constant $C$ for which
	\begin{equation} \label{eq: Dudley}
	\gamma_1(T,d) \leq C \int_0^\infty \log N(T,d,u) du.
	\end{equation}
	
	Recall that $\mathcal{S}_{2\theta/\pi}$ is the set of all $\braces{0,1}$ vectors with fewer than $2\theta/\pi$ ones. For convenience, let us assume that $2m\theta/\pi$ is an integer. We then have the inclusion
	\begin{equation*}
	\mathcal{S}_{2\theta/\pi} \subset \bigcup_{I \in \mathcal{I}} [0,1]^I,
	\end{equation*}
	where $\mathcal{I}$ is the collection of all subsets of $[m]$ of size $2m\theta/\pi$, and $[0,1]^I$ denotes the unit cube in the coordinate set $I$.
	We may then also write
	\begin{equation*}
	\mathcal{S}_{2\theta/\pi} \cross B^n_2 \subset \bigcup_{I \in \mathcal{I}} \paren{[0,1]^I \cross B^n_2}.
	\end{equation*}
	
	Note that a union of covers for each $[0,1]^I \cross B^n_2$ gives a cover for $\mathcal{S}_{2\theta/\pi} \cross B^n_2$. This, together with the symmetry of $\norm{\cdot}_\infty$ with respect to permutation of the coordinates gives
	\begin{equation*}
	N(\mathcal{S}_{2\theta/\pi} \cross B^n_2,\tripnorm{\cdot}_1,u)  \leq \abs{\mathcal{I}} \cdot N([0,1]^I \cross B^n_2,\tripnorm{\cdot}_1,u)
	\end{equation*}
	for some fixed index set $I$.
	
	We next generalize the notion of covering numbers slightly. Given two sets $T$ and $K$, we let $N(T,K)$ denote the number of translates of $K$ needed to cover the set $T$. It is easy to see that we have $N(T,d,u) = N(T,uB_d)$, where $B_d$ is the unit ball with respect to the metric $d$. Since the unit ball for $\tripnorm{\cdot}_1$ is $B_\infty^m \cross B_2^n$, we therefore have
	\begin{align*}
	N([0,1]^I \cross B^n_2,\tripnorm{\cdot}_1,u) & = N([0,1]^I \cross B^n_2, u(B_\infty^m \cross B_2^n)) \\
	& \leq  N(B^{(2\theta/\pi)m}_\infty \cross B^n_2, u(B^{(2\theta/\pi)m}_\infty\cross B_2^n) ).
	\end{align*}
	
	Such a quantity can be bounded using a volumetric argument. Generally, for any centrally symmetric convex body $K$ in $\R^n$, we have (see Corollary 4.1.15 in \cite{ShiriArtstein-AvidanTelAvivUniversityTelAviv2015})
	\begin{equation} \label{eq: volume bound for covering number}
	N(K,uK) \leq (3/u)^n.
	\end{equation}
	This implies that 
	\begin{equation*}
	\log N([0,1]^I \cross S^{n-1},\tripnorm{\cdot}_1,u) \leq \log(3/u)((2\theta/\pi)m+n).
	\end{equation*}
	Finally, observe that
	\begin{equation*}
	\log\abs{\mathcal{I}} = \log \binom{m}{(2\theta/\pi)m} \leq (2\theta/\pi)m\log(e\pi/2\theta).
	\end{equation*}
	
	We can thus plug these last two bounds into \eqref{eq: Dudley}, noting that the integrand is zero for $u \geq 1$ to get
	\begin{align*}
	\gamma_1(\mathcal{S}_{2\theta/\pi}\cross B^n_2,\tripnorm{\cdot}_1) & \leq C \int_0^1 (2\theta/\pi)m\log(e\pi/2\theta) + \log(3/u)((2\theta/\pi)m+n) du \\
	& \leq C\paren{(2\theta/\pi)\log(\pi/2\theta)m + n}
	\end{align*}
	as was to be shown.
\end{proof}

\begin{lemma}
	We may bound the $\gamma_2$ functional of $\mathcal{S}_{2\theta/\pi}\cross B^n_2$ by
	\begin{equation*}
	\gamma_2(\mathcal{S}_{2\theta/\pi}\cross B^n_2,\tripnorm{\cdot}_2) \leq C\sqrt{2\theta/\pi}\paren{m + \sqrt{mn}}.
	\end{equation*}
\end{lemma}

\begin{proof}
	Since $\alpha = 2$, we may appeal directly to the theory of Gaussian complexity \cite{Vershynin}. However, since we have already introduced some of the theory of metric entropy in the previous lemma, we might as well continue down this path. In this case, we have the Dudley bound
	\begin{equation} \label{eq: Dudley 2}
	\gamma_2(T,d) \leq C \int_0^\infty \sqrt{\log N(T,d,u)} du
	\end{equation}
	for any metric space $(T,d)$.
	
	Observe that the unit ball for $\tripnorm{\cdot}_2$ is $B^m_2 \cross (2m\theta/\pi)^{-1/2}B^n_2$. On the other hand, we conveniently have
	\begin{equation*}
	\mathcal{S}_{2\theta/\pi} \cross B^n_2 \subset \sqrt{2m\theta/\pi}B^m_2 \cross B^n_2.
	\end{equation*}
	As such, we have
	\begin{align*}
	N(\mathcal{S}_{2\theta/\pi} \cross B^n_2,\tripnorm{\cdot}_2,u) & \leq 	N(\sqrt{2m\theta/\pi}B^m_2 \cross B^n_2,\tripnorm{\cdot}_2,u) \\
	& = N(\sqrt{2m\theta/\pi}B^m_2 \cross B^n_2,u (B^m_2 \cross (2m\theta/\pi)^{-1/2}B^n_2)) \\
	& = N(T,(2m\theta/\pi)^{-1/2}uT),
	\end{align*}
	where $T = \sqrt{2m\theta/\pi}B^m_2 \cross B^n_2$.
	
	Plugging this into \eqref{eq: Dudley 2} and subsequently using the volumetric bound \eqref{eq: volume bound for covering number}, we get
	\begin{align*}
	\gamma_2(\mathcal{S}_{2\theta/\pi}\cross B^n_2,\tripnorm{\cdot}_2) & \leq C \int_0^\infty \sqrt{\log N(T,(2m\theta/\pi)^{-1/2}uT)} du \\
	& = C\sqrt{2m\theta/\pi}\int_0^\infty \sqrt{\log N(T,uT)} du \\
	& \leq C \sqrt{2m\theta/\pi}\sqrt{m+n},
	\end{align*}
	which is clearly equivalent to the bound that we want.
\end{proof}

At this stage, we can put everything together to bound the supremum of our stochastic process.

\begin{theorem}[Bound on supremum of $(Y_{s,v})$] \label{thm: bound on sup of Y}
	Let $(Y_{s,v})$ be the process defined in \eqref{def: Y process}. Let $0 < \delta < 1/e$, let $\theta$ be an acute angle, and suppose $m \geq \max\braces{n,\log(1/\delta)\pi/2\theta}$. Then with probability at least $1-\delta$, the supremum of the process satisfies
	\begin{equation} \label{eq: bound on sup of Y}
	\sup_{s \in \mathcal{S}_{2\theta/\pi},v \in B^n_2} Y_{s,v} \leq C\sqrt{2\theta/\pi}\cdot m
	\end{equation}
\end{theorem}

\begin{proof}
	It is easy to see that we have
	\begin{equation*}
	\diam\paren{\mathcal{S}_{2\theta/\pi} \cross B^n_2,\tripnorm{\cdot}_1} = 2,
	\end{equation*}
	and
	\begin{equation*}
	\diam\paren{\mathcal{S}_{2\theta/\pi} \cross B^n_2,\tripnorm{\cdot}_2} = 2\sqrt{2m\theta/\pi}.
	\end{equation*}
	Also observe that we have $Y_{s,0} = 0$ for any $s \in \mathcal{S}_{2\theta/\pi}$.
	
	Using these, together with the previous two lemmas bounding the $\gamma_1$ and $\gamma_2$ functionals, we may apply Lemma \ref{lem: sup bound for mixed tail process} to see that
	\begin{equation*}
	\sup_{s \in \mathcal{S}_{2\theta/\pi},v \in B^n_2} Y_{s,v} \leq C\paren{\paren{(2\theta/\pi)\log(\pi/2\theta)m + n} + \sqrt{2\theta/\pi}\paren{m + \sqrt{mn} } + u + \sqrt{u} \sqrt{2m\theta/\pi}}.
	\end{equation*}
	with probability at least $1-e^{-u}$.
	
	Using our assumptions on $m$, we may simplify this bound to obtain \eqref{eq: bound on sup of Y}.
\end{proof}

Finally, we show that $\frac{1}{m}\sum_{i=1}^m a_i a_i^T$ is well-behaved.

\begin{lemma} \label{lem: empirical covariance concentration}
	Let $\delta >0$. Then if $m \geq C(n+\sqrt{\log(1/\delta)})$, with probability at least $1-\delta$, we have
	\begin{equation*}
	\norm*{\frac{n}{m}\sum_{i=1}^ma_ia_i^T - I_n} \leq 0.1
	\end{equation*}
\end{lemma}

\begin{proof}
	Note, as before, that the $\sqrt{n}a_i$'s are isotropic subgaussian random variables with subgaussian norm bounded by an absolute constant. The claim then follows immediately from Theorem 5.39 in \cite{Vershynin2011a}, which itself is proved using Bernstein and a simple net argument.
\end{proof}

\begin{theorem}[Finite measurement sets satisfy $\ACW$ condition] \label{thm: finite measurement sets satisfy ACW}
	There is some $\theta_0 > 0$ and an absolute constant $C$ such that for all angles $0 < \theta \leq \theta_0$, for all dimensions $n$, and any $\delta > 0$, if $m$ satisfies
	\begin{equation} \label{eq: assumption on m}
	m \geq C(\pi/2\theta)^2(n\log(m/n) + \log(1/\delta)), 
	\end{equation}
	then with probability at least $1-\delta$, the measurement set $A$ comprising $m$ independent random vectors drawn uniformly from $S^{n-1}$ satisfies the $\ACW(\theta,\alpha)$ condition with $\alpha = 1/2$.
\end{theorem}

\begin{proof}
	Fix $n, \delta > 0$. Choose $\theta_0$ such that the constant $C$ in the statement in Theorem \ref{thm: bound on sup of Y} satisfies $C\sqrt{2\theta_0/\pi} \leq 0.1$. Fix $0 < \theta \leq \theta_0$, and let $\Omega_1$, $\Omega_2$, and $\Omega_3$ denote the good events in Lemma \ref{lem: uniform conc over wedges}, Theorem \ref{thm: bound on sup of Y}, and Lemma \ref{lem: empirical covariance concentration} with this choice of $\theta$. Whenever $m$ satisfies our assumption \eqref{eq: assumption on m}, the intersection of these events occurs with probability at least $1-3\delta$ by the union bound.
	
	Let us condition on being in the intersection of these events. For any wedge $W \in \mathcal{W}_\theta$ (i.e of angle less than $\theta$), Lemma \ref{lem: uniform conc over wedges} tells us that its associated selector vector satisfies $s_{W,A} \in \mathcal{S}_{2\theta/\pi}$ (i.e. that it has at most $2m\theta/\pi$ ones,). By Theorem \ref{thm: bound on sup of Y} and our assumption on $\theta_0$, we then have
	\begin{equation*}
	\lambda_{\max}\paren*{\frac{1}{m}\sum_{i=1}^m a_ia_i^T1_W(a_i)} \leq \frac{1}{nm}\sup_{s \in \mathcal{S}_{2\theta/\pi}, v \in B^n_2} Y_{s,v} \leq \frac{0.1}{n}.
	\end{equation*}
	
	On the other hand, Lemma \ref{lem: empirical covariance concentration} guarantees that 
	\begin{equation*}
	\lambda_{\min}\paren*{\frac{1}{m}\sum_{i=1}^m a_ia_i^T} \geq \frac{0.9}{n}.
	\end{equation*}
	Combining these, we get
	\begin{equation*}
	\lambda_{\min}\paren*{\frac{1}{m}\sum_{i=1}^m a_ia_i^T - \frac{4}{m}\sum_{i=1}^m a_ia_i^T1_W(a_i)} \geq \frac{1}{2n},
	\end{equation*}
	which was to be shown.
\end{proof}

\section{Proof and discussion of Theorem \ref{thm: guarantee for algorithm}} \label{sec: proof of guarantee}

We restate the theorem here for convenience.

\begin{theorem} \label{thm: repeated main theorem}
	Fix $\epsilon > 0$, $0 < \delta_1 \leq 1/2$, and $0 < \delta,\delta_2 \leq 1$. There are absolute constants $C, c > 0$ such that if
	\begin{equation*}
	m \geq C(n\log(m/n) + \log(1/\delta)),
	\end{equation*}
	then with probability at least $1-\delta$, $m$ sampling vectors selected uniformly and independently from the unit sphere $S^{n-1}$ form a set such that the following holds: Let $x \in \R^n$ be a signal vector and let $x_0$ be an initial estimate satisfying $\norm{x_0-x}_2 \leq c\sqrt{\delta_1}\norm{x}_2$. Then for any $\epsilon > 0$, if
	\begin{equation*}
	K \geq 2(\log(1/\epsilon) + \log(2/\delta_2))n,
	\end{equation*}
	then the $K$-th step randomized Kaczmarz estimate $x_K$ satisfies $\norm{x_K-x}_2^2 \leq \epsilon \norm{x_0 - x}_2^2$ with probability at least $1-\delta_1-\delta_2$.
\end{theorem}

\begin{proof}
	Let $A$ be our $m$ by $n$ measurement matrix. By Theorem \ref{thm: finite measurement sets satisfy ACW}, there is an angle $\theta_0$, and a constant $C$ such that for $m \geq C(n\log(m/n) + \log(1/\delta))$, $A$ is $\ACW(\theta_0,1/2)$ with probability at least $1-\delta$.
	
	We can then use Corollary \ref{cor: linear convergence for ACW measure, high probability} to guarantee that with probability at least $1-\delta_1-\delta_2$, running the randomized Kaczmarz update $K$ times gives an estimate $x_K$ satisfying
	\begin{equation*}
	\norm{x_K-x}_2^2 \leq \epsilon\norm{x_0-x}_2^2.
	\end{equation*}
	This completes the proof of the theorem.
\end{proof}

Inspecting the statement of the theorem, we see that we can make the failure probability $\delta$ as small as possible by making $m$ large enough. Likewise, we can do the same with $\delta_2$ by adjusting $K$. Proposition \ref{prop: initialization guarantee} shows that we can also make $\delta_2$ smaller by increasing $m$. However, while the dependence of $m$ and $K$ on $\delta$ and $\delta_2$ respectively is logarithmic, the dependence of $m$ on $\delta_1$ is \textit{polynomial} (we need $m \gtrsim 1/\delta_1^2$). This is rather unsatisfactory, but can be overcome by a simple ensemble method. We encapsulate this idea in the following algorithm.

\begin{algorithm}[H]
	\caption{{\sc Ensemble Randomized Kaczmarz}}
	\begin{algorithmic}[1]              
		\REQUIRE Measurements $b_1,\ldots,b_m$, sampling vectors $a_1,\ldots,a_m$, relative error tolerance $\epsilon$, iteration count $K$, trial count $L$.
		\ENSURE An estimate $\hat{x}$ for $x$.
		\STATE Obtain an initial estimate $x_0$ using the truncated spectral initialization method (see Appendix \ref{sec: initialization}).
		\STATE \textbf{for} $l=1,\ldots,L$, run $K$ randomized Kaczmarz update steps starting from $x_0$ to obtain an estimate $x_K^{(l)}$.
		\STATE \textbf{for} $l=1,\ldots,L$, \textbf{do}
		\STATE \quad \textbf{if} $\abs{B(x_K^{(l)},2\sqrt{\epsilon}) \cap \braces{x_K^{(1)},\ldots,x_K^{(L)}}} \geq L/2$
		\STATE \quad\quad \textbf{return} $\hat{x} := x_K^{(l)}$.
	\end{algorithmic}
	\label{alg: ensemble RK}
\end{algorithm}

\begin{proposition}[Guarantee for ensemble method]
	Given the assumptions of Theorem \ref{thm: repeated main theorem}, further assume that $\delta_1 + \delta_2 \leq 1/3$. For any $\delta' > 0$, there is an absolute constant $C$ such that if $L \geq C\log(1/\delta')$, then the estimate $\hat{x}$ given by Algorithm \ref{alg: ensemble RK} satisfies $\norm{\hat{x}-x}_2^2 \leq 9\epsilon \norm{x_0-x}_2^2$ with probability at least $1-\delta'$.
\end{proposition}

\begin{proof}
	For $1 \leq l \leq L$, let $\chi_l$ be the indicator variable for $\norm{x_K^{(l)}-x}_2^2 \leq \epsilon \norm{x_0-x}_2^2$. Then $\chi_1,\ldots,\chi_L$ are i.i.d. Bernoulli random variables each with success probability at least $2/3$. Let $I$ be the set of indices $l$ for which $\chi_l = 1$. Using a Chernoff bound \cite{Vershynin}, we see that with probability at least $1-e^{-cL}$, $\abs{I} \geq L/2$. Now let $I'$ be the set of indices for which $\abs{B(x_K^{(l)},2\epsilon) \cap \braces{x_K^{(1)},\ldots,x_K^{(L)}}} \geq L/2$. Observe that for all $l,l' \in I$, we have
	\begin{equation*}
	\norm{x_K^{(l)}- x_K^{(l')}}_2 \leq \norm{x_K^{(l)}- x}_2 + \norm{x- x_K^{(l')}}_2 \leq 2\sqrt{\epsilon}.
	\end{equation*}
	This implies that $I \subset I'$, so $I' \neq \emptyset$. Furthermore, for all $l' \in I'$, there is $l \in I$ for which $\norm{x_K^{(l)}-x_K^{(l')}}_2 \leq 2\sqrt{\epsilon}$. As such, we have
	\begin{equation*}
	\norm{x_K^{(l')}-x}_2 \leq \norm{x_K^{(l')}- x_K^{(l)}}_2 + \norm{x_K^{(l)}-x}_2 \leq 3\sqrt{\epsilon}.
	\end{equation*}
	Now, observe that the estimate $\hat{x}$ returned by Algorithm \ref{alg: ensemble RK} is precisely some $x_K^{(l')}$ for which $l' \in I'$. This shows that on the good event, we indeed have $\norm{\hat{x}-x}_2^2 \leq 9\epsilon \norm{x_0-x}_2^2$. By our assumption on $L$, we see that the failure probability is bounded by $\delta'$.
\end{proof}

In practice however, the ensemble method is not required. Numerical experiments show that the randomized Kaczmarz method always eventually converges from any initial estimate.

\section{Extensions}

\subsection{Arbitrary initialization}

In order to obtain a convergence guarantee, we used a truncated spectral initialization to obtain an initial estimate before running randomized Kaczmarz updates. Since the number of steps that we require is only linear in the dimension, and each step requires only linear time, the iteration phase of the algorithm only requires $O(n^2)$ time, and furthermore does not need to see all the data in order to start running.

The spectral initialization on the other hand requires one to see all the data. Forming the matrix from which we obtain the estimate involves adding $m$ rank 1 matrices, and hence naively requires $O(mn^2)$ time. There is hence an incentive to do away with this step altogether, and ask whether the randomized Kaczmarz algorithm works well even if we start from an arbitrary initialization.

We have some numerical evidence that this is indeed true, at least for real Gaussian measurements. Unfortunately, we do not have any theoretical justification for this phenomenon, and it will be interesting to see if any results can be obtained in this direction.

\subsection{Complex Gaussian measurements}

We have proved our main results for measurement systems comprising random vectors drawn independently and uniformly from the sphere, or equivalently, for real Gaussian measurements. These are not the measurement sets that are used in practical applications, which often deal with imaging and hence make use of complex measurements.

While most theoretical guarantees for phase retrieval algorithms are in terms of real Gaussian measurements, some also hold for complex Gaussian measurements, even with identical proofs. This is the case for \textit{PhaseMax} \cite{Cand??s2013} and for \text{Wirtnger flow} \cite{Candes2015}. We believe that a similar situation should hold for the randomized Kaczmarz method, but are not yet able to recalibrate our tools to handle the complex setting.

It is easy to adapt the randomized Kaczmarz update formula \eqref{def: randomized Kaczmarz phase retrieval} itself: we simply replace the sign of $\inprod{a_{r(k)},x_{k-1}}$ with its phase (i.e. $\frac{\inprod{a_{r(k)},x_{k-1}}}{\abs{\inprod{a_{r(k)},x_{k-1}}}}$). Numerical experiments also show that convergence does occur for complex Gaussian measurements (and even CDP measurements) \cite{Wei2015}. Nonetheless, in trying to adapt the proof to this situation, we meet an obstacle at the first step: when computing the error term, we can no longer simply sum up the influence of ``bad measurements'' as we did in Lemma \ref{lem: expected decrement}. Instead, \textit{every} term contributes an error that scales with the phase difference
\begin{equation*}
\frac{\inprod{a_i,z}}{\abs{\inprod{a_i,z}}} - \frac{\inprod{a_i,x}}{\abs{\inprod{a_i,x}}}.
\end{equation*}

We leave it to future work to prove convergence in this setting, whether by adapting our methods, or by proposing completely new ones.

\subsection{Deterministic constructions of measurement sets}

The theory that we have developed in this paper does not apply solely to Gaussian measurements, and generalizes to any measurement sets that satisfy the $\ACW$ condition that we introduced in Section \ref{sec: ACW condition}. It will be interesting to investigate what natural classes of measurement sets satisfy this condition.

\subsection*{Acknowledgments}

Y.T. was partially supported by the Juha Heinonen Memorial Graduate Fellowship at the University of Michigan. R.V. was partially supported by NSF Grant DMS 1265782 and U.S. Air Force Grant FA9550-18-1-0031. Y.T. would like to thank Halyun Jeong for insightful discussions on this topic. We would also like to thank the anonymous reviewers for their many helpful comments.

\nocite{*}
\bibliographystyle{acm}
\bibliography{Projects-Kaczmarz_PR}

\appendix

\section{Growth functions and VC dimension} \label{sec: growth fns and VC dim}

In this section, we define growth functions and VC dimension. We also state some standard results on these topics that we require for our proofs in Section \ref{sec: ACW condition}. We refer the interested reader to \cite{Mohri2012} for a more in-depth exposition on these topics.

Let $\mathcal{X}$ be a set and $\mathcal{C}$ be a family of subsets of $\mathcal{X}$. For a given set $C \in \mathcal{C}$, we slightly abuse notation and identify it with its indicator function $1_C \colon \mathcal{X} \to \braces{0,1}$.

The \textit{growth function} $\Pi_{\mathcal{C}}\colon \mathbb{N} \to \R$ of $\mathcal{C}$ is defined via
\[
\Pi_{\mathcal{C}}(m) := \max_{x_1,\ldots,x_m \in \mathcal{X}} \abs*{\braces*{\paren{C(x_1),C(x_2),\ldots,C(x_m)} ~\colon~ C \in \mathcal{C}}}.
\]

Meanwhile, the \textit{VC dimension} of $\mathcal{C}$ is defined to be the largest integer $m$ for which $\Pi_{\mathcal{C}}(m) = 2^m$. These two concepts are fundamental to statistical learning theory. The key connection between them is given by the Sauer-Shelah lemma.

\begin{lemma}[Sauer-Shelah, Corollary 3.3 in \cite{Mohri2012}] \label{lem: sauer-shelah}
	Let $\mathcal{C}$ be a collection of subsets of VC dimension $d$. Then for all $m \geq d$, have
	\[
	\Pi_{\mathcal{C}}(m) \leq \paren*{\frac{em}{d}}^d.
	\]
\end{lemma}

The reason why we are interested in the growth function of a family of subsets $\mathcal{C}$ is because we have the following guarantee for the uniform convergence for the empirical measures of sets belonging to $\mathcal{C}$.

\begin{proposition}[Uniform deviation, Theorem 2 in \cite{Vapnik1971a}] \label{prop: VC uniform convergence}
	Let $\mathcal{C}$ be a family of subsets of a set $\mathcal{X}$. Let $\mu$ be a probability measure on $\mathcal{X}$, and let $\hat{\mu}_m := \frac{1}{m} \sum_{i=1}^m \delta_{X_i}$ be the empirical measure obtained from $m$ independent copies of a random variable $X$ with distribution $\mu$. For every $u$ such that $m \geq 2/u^2$, the following deviation inequality holds:
	\begin{equation}
	\P\paren{\sup_{C \in \mathcal{C}} \abs{\hat{\mu}_m(C) - \sigma(C)} \geq u } \leq 4\Pi_{\mathcal{C}}(2m)\exp(-mu^2/16).
	\end{equation}
\end{proposition}

We now state and prove two simple claims.

\begin{claim} \label{clm: VC dim for hemispheres}
	Let $\mathcal{C}$ be the collection of all hemispheres in $S^{n-1}$. Then the VC dimension of $\mathcal{C}$ is bounded from above by $n+1$.
\end{claim}

\begin{proof}
	It is a standard fact from statistical learning theory \cite{Mohri2012} that the VC dimension of half-spaces in $\R^n$ is $n+1$. Since $S^{n-1}$ is a subset of $\R^n$, the claim follows by the definition of VC dimension.
\end{proof}

\begin{claim} \label{clm: growth function bound for symmetric difference}
	Let $\mathcal{C}$ and $\mathcal{D}$ be two collections of functions from a set $\mathcal{X}$ to $\braces{0,1}$. Using $\triangle$ to denote symmetric difference, we define
	\begin{equation} \label{def: symmetric difference}
	\mathcal{C}\triangle\mathcal{D} := \braces{C\triangle D \st C \in \mathcal{C}, D \in \mathcal{D}}.
	\end{equation}
	Then the growth function $\Pi_{\mathcal{C}\triangle\mathcal{D}}$ of $\mathcal{C}\triangle\mathcal{D}$ satisfies $\Pi_{\mathcal{C}\triangle\mathcal{D}}(m) \leq \Pi_{\mathcal{C}}(m)\cdot \Pi_{\mathcal{D}}(m)$ for all $m \in \Z_+$.
\end{claim}

\begin{proof}
	Fix $m$, and points $x_1,\ldots,x_m \in \mathcal{X}$. Then every possible configuration $\paren{f(x_1),f(x_2),\ldots,f(x_m)}$ arising from some $f \in \mathcal{C}\triangle\mathcal{D}$ is the point-wise symmetric difference
	\begin{equation*}
	\paren{f(x_1),f(x_2),\ldots,f(x_m)} = \paren{C(x_1),C(x_2),\ldots,C(x_m)}\triangle \paren{D(x_1),D(x_2),\ldots,D(x_m)}
	\end{equation*}
	of configurations arising from some $C \in \mathcal{C}$ and $D \in \mathcal{D}$. By the definition of growth functions, there are at most $\Pi_{\mathcal{C}}(m)\cdot \Pi_{\mathcal{D}}(m)$ pairs of these configurations, from which the bound follows.
\end{proof}

\begin{remark}
	There is an extensive literature on how to bound the VC dimension of concept classes that arise from finite intersections or unions of those from a known collection of concept classes, each of which has bounded VC dimension. We won't require this much sophistication here, and refer the reader to \cite{Blumer1989} for more details.
\end{remark}

\section{Initialization} \label{sec: initialization}

Several different schemes have been proposed for obtaining initial estimates for \textit{PhaseMax} and gradient descent methods for phase retrieval. Surprisingly, these are all spectral in nature: the initial estimate $x_0$ is obtained as the leading eigenvector to a matrix that is constructed out of the sampling vectors $a_1,\ldots,a_m$ and their associated measurements $b_1,\ldots,b_m$ \cite{Candes2015,Chen2015,Zhang2016,Wang2017}.

There seems to be empirical evidence, at least for Gaussian measurements, that the best performing method is the \textit{orthogonality-promoting method} of \cite{Wang2017}. Nonetheless, for any given relative error tolerance, all the methods seem to require sample complexity of the same order. Hence, we focus on the \textit{truncated spectral method} of \cite{Chen2015} for expositional clarity, and refer the reader to the respective papers on the other methods for more details.

The truncated spectral method initializes $x_0 := \lambda_0 \tilde{x}_0$, where $\lambda_0 = \sqrt{\frac{1}{m}\sum_{i=1}^mb_i^2}$, and $\tilde{x}_0$ is the leading eigenvector of
\begin{equation*}
Y = \frac{1}{m}\sum_{i=1}^{m}b_i^2a_ia_i^T1(b_i \leq 3\lambda_0).
\end{equation*}
Note that when constructing $Y$, we sum up only those sampling vectors whose corresponding measurements satisfy $b_i \leq 3\lambda_0$. The point of this is to remove the influence of unduly large measurements, and allow for good concentration estimates, as we shall soon demonstrate.

Suppose from now on that the $a_i$'s are independent standard Gaussian vectors. In \cite{Chen2015}, the authors prove that with probability at least $1-\exp(-\Omega(m))$, we have $\norm{\tilde{x}_0 - x}_2 \leq \epsilon \norm{x}_2$ for any fixed relative error tolerance $\epsilon$ (see their Proposition 3). They do not, however, examine the dependence of the probability bound on $\epsilon$. Nonetheless, by examining the proof more carefully, we can make this dependence explicit. In doing so, we obtain the following proposition.

\begin{proposition}[Relative error guarantee for initialization] \label{prop: initialization guarantee}
	Let $a_1,\ldots,a_m$, $b_1,\ldots,b_m$, $Y$ and $x_0$ be defined as in the preceding discussion. Fix $\epsilon > 0$ and $0 < \delta < 1$. Then with probability at least $1-\delta$, we have $\norm{x_0-x}_2 \leq \epsilon \norm{x}_2$ so long as $m \geq C\paren{\log(1/\delta)+n}/\epsilon^2$.
\end{proposition}

\begin{proof}
	We simply make the following observations while following the proof of Proposition 3 in \cite{Chen2015}. First, since all quantities are 2-homogeneous in $\norm{x}_2$, we may assume without loss of generality that $\norm{x}_2 = 1$. Next, there is some absolute constant $c$ such that if we define $Y_1$ and $Y_2$ by choosing $\gamma_1 = 3+ c\epsilon$, $\gamma_2 = 3 - c\epsilon$, we have the bound $\norm{\E Y_1 - \E Y_2} \leq C\epsilon$. Note also that the deviation estimates $\norm{Y_1-\E Y_1}$, $\norm{Y_2-\E Y_2}$ are bounded by $C\epsilon$ given our assumptions on $m$. This implies that with high probability,
	\begin{equation*}
	\norm{Y-\beta_1 xx^T - \beta_2 I_n} \leq C\epsilon.
	\end{equation*}
	Adjust our constants so that $C$ in the last equation is bounded by $\beta_1 - \beta_2$. We may then apply Davis-Kahan \cite{Davis1970a} to get
	\begin{equation*}
	\norm{\tilde{x}_0 - x}_2 \leq \frac{\norm{Y-\beta_1 xx^T - \beta_2 I_n}}{\beta_1 - \beta_2} \leq \epsilon
	\end{equation*}
	as we wanted.
\end{proof}

By examining the proof carefully, the astute reader will observe that the crucial properties that we used were the rotational invariance of the $a_i$'s (to compute the formulas for $\E Y_1$ and $\E Y_2$) and their subgaussian tails (to derive the deviation estimates). These properties also hold for sampling vectors that are uniformly distributed on the sphere. As such, a more lengthly and tedious calculation can be done to show that the guarantee also holds for such sampling vectors. If the reader has any residual doubt, perhaps this can be assuaged by noting that a uniform sampling vector and its associated measurement $(a_i,b_i)$ can be turned into an honest real Gaussian vector by multiplying both quantities by an independent $\chi^2$ random variable with $n$ degrees of freedom.

%\begin{lemma}
%	The matrix $\E[Y] = \beta \tilde{x}\tilde{x}^T + \gamma I_n$.
%\end{lemma}
%
%\begin{proof}
%	Assume without loss of generality that $\tilde{x} = e_1$. We first compute concentration bounds on $\lambda_0$. We have
%	
%\end{proof}
%
%\begin{lemma}
%	\begin{equation*}
%	\P\paren{\abs{Y-\E Y } > \epsilon} \leq 2\exp\paren*{n\log 9 -m \min\braces{\frac{\epsilon^2}{\norm{x}_2^4},\epsilon/\norm{x}_2^2}}
%	\end{equation*}
%\end{lemma}

\end{document}